\documentclass[oneside,11pt]{amsart}

\usepackage[english]{babel}
\usepackage{amsmath}
\usepackage{amssymb}
\usepackage{amsthm}
\usepackage{latexsym}
\usepackage{mathtools}
\usepackage{etexcmds} % TeX Live 2019 needs this for thmtools to work
\usepackage{thmtools}
\usepackage{amsfonts}
\usepackage{mathrsfs}
\usepackage{textcomp}
\usepackage[T1]{fontenc}
\usepackage{graphicx}
\usepackage{setspace}
\usepackage{nicefrac}
\usepackage{indentfirst}
\usepackage{enumerate}
\usepackage{wasysym}
\usepackage[normalem]{ulem}
\usepackage{upgreek}
\usepackage[pdfpagelabels,hyperindex=false]{hyperref}
\usepackage{paralist}
\usepackage{xcolor}
\hypersetup{
    colorlinks,
    linkcolor={red!50!black},
    citecolor={green!50!black},
    urlcolor={blue!80!black}
}
\usepackage{etoolbox}
\usepackage{mathdots}
\usepackage{wrapfig}
\usepackage{floatflt}
\usepackage{tensor} 
\usepackage{parskip}
\usepackage{lscape}
\usepackage{stmaryrd}
\usepackage{ytableau}
\usepackage{pifont}

\usepackage{tikz}
\usetikzlibrary{arrows,matrix}
\usetikzlibrary{positioning}
\usetikzlibrary{decorations}

\usepackage[all,cmtip]{xy}
\usepackage[labelfont=bf, font={small}]{caption}[2005/07/16]

\usepackage[top = 3cm, bottom = 3cm, left=3cm, right=3cm]{geometry}

\makeatletter
\@namedef{subjclassname@2020}{%
  \textup{2020} Mathematics Subject Classification}
\makeatother

\newcommand{\deggeq}{\trianglerighteq}
\newcommand{\degleq}{\trianglelefteq}

\newcommand{\Mamu}{\mathbf{MaMu}}
\newcommand{\EMamu}{\mathbf{EMaMu}}
\newcommand{\multiEMamu}{\mathbf{multiEMaMu}}
\newcommand{\Dome}{\mathbf{Dome}}

\newcommand{\triv}{\mathrm{triv}}
\DeclareMathOperator{\GL}{\mathrm{GL}}
\newcommand{\vvirg}{,\dots,}
\newcommand{\ootimes}{\otimes \cdots \otimes }
\newcommand{\ttimes}{\times \cdots \times }
\newcommand{\textsum}{{\textstyle \sum}}
\newcommand{\textbinom}[2]{{\textstyle \binom{#1}{#2}}}
\newcommand{\textfrac}[2]{{\textstyle \frac{#1}{#2}}}
\newcommand{\calZ}{\mathcal{Z}}
\newcommand{\calO}{\mathcal{O}}
\newcommand{\bbC}{\mathbb{C}}
\newcommand{\bbN}{\mathbb{N}}
\newcommand{\bfu}{\mathbf{u}}

\newcommand{\bfj}{\mathbf{j}}
\newcommand{\bfk}{\mathbf{k}}
\newcommand{\bfn}{\mathbf{n}}
\newcommand{\bfZ}{\mathbf{Z}}
\newcommand{\bigboxtimes}{{\scalebox{1.2}{$\boxtimes$}}}
\newcommand{\dotitem}{\item[$\cdot$]}

\renewcommand{\bar}[1]{\overline{#1}}

\newtheorem{theorem}{Theorem}[section]

\newtheorem{lemma}[theorem]{Lemma}
\newtheorem{proposition}[theorem]{Proposition}
 \newtheorem{claim}{Claim}
 
\theoremstyle{definition}

\newtheorem{example}[theorem]{Example}

\theoremstyle{remark}

\newcommand{\rmR}{\mathrm{R}}
\newcommand{\uR}{\underline{\mathrm{R}}}
\renewcommand{\phi}{\varphi}
\newcommand{\eps}{\varepsilon}
\newcommand{\aR}{\uwave{\rmR}}

\title[Border rank non-additivity]{Border rank non-additivity for higher order tensors}

\author[M. Christandl, F. Gesmundo, M. Micha{\l}ek, J. Zuiddam]{Matthias Christandl, Fulvio Gesmundo, Mateusz Micha{\l}ek, Jeroen Zuiddam}

\address[M. Christandl]{QMATH, Dept. Math. Sciences, University of Copenhagen, Universitetsparken 5, 2100 Copenhagen O., Denmark}
\email[M. Christandl]{christandl@math.ku.dk}

\address[F. Gesmundo]{QMATH, Dept. Math. Sciences, University of Copenhagen, Universitetsparken 5, 2100 Copenhagen O., Denmark -- (current) Max Planck Institute for Mathematics in the Sciences, Inselstrasse 22, 04103 Leipzig, Germany}
\email[F. Gesmundo]{fulvio.gesmundo@mis.mpg.de}

\address[M. Micha{\l}ek]{Max Planck Institute for Mathematics in the Sciences, Inselstrasse 22, 04103 Leipzig, Germany -- (current) Department of Mathematics and Statistics, University of Konstanz, Konstanz, Germany}
\email[M. Micha{\l}ek]{mateusz.michalek@uni-konstanz.de}

\address[J. Zuiddam]{Institute for Advanced Study, 1 Einstein Drive, Princeton, NJ -- (current) Courant Institute of Mathematical Sciences, New York University, 251 Mercer Street, 10012-1185, New York, NY, USA}
\email[J. Zuiddam]{jzuiddam@nyu.edu}

\begin{document}
\begin{abstract}
Whereas matrix rank is additive under direct sum, in 1981 Sch\"onhage showed that one of its generalizations to the tensor setting, tensor border rank, can be strictly subadditive for tensors of order three. Whether border rank is additive for higher order tensors has remained open. In this work, we settle this problem by providing analogues of Sch\"onhage's construction for tensors of order four and higher. Sch\"onhage's work was motivated by the study of the computational complexity of matrix multiplication; we discuss implications of our results for the asymptotic rank of higher order generalizations of the matrix multiplication tensor. 
\end{abstract}
 \maketitle
\section{Introduction}

Let $V_1 \vvirg V_k$ be finite dimensional complex vector spaces and let $T \in V_1 \ootimes V_k$ be a tensor. The \emph{tensor rank} of $T$ is defined as
\[
 \rmR(T) = \min \left\{ r: T = \textsum_{i=1}^r v_1^{(i)} \ootimes v_k^{(i)}  \text{ for some } v_j^{(i)} \in V_j \right\}.
\]
Tensor rank generalizes matrix rank: indeed, if $k=2$, the tensor rank of $T \in V_1 \otimes V_2$ coincides with the rank of the corresponding linear map $T: V_1^* \to V_2$.

The \emph{tensor border rank} (or simply \emph{border rank}) of $T$ is defined as
\[
 \uR(T) = \min \left\{ r : T = \lim_{\eps \to 0} T_\eps \text{ with } \rmR(T_\eps) = r \text{ for } \eps\neq 0 \right\},
\]
where the limit is taken in the Euclidean topology of $V_1 \ootimes V_k$. One immediately has $\uR(T) \leq \rmR(T)$; for $k \geq 3$, there are examples where the inequality is strict.

The study of geometric properties of tensor rank and border rank has a long history dating back to more than a century ago \cite{Sylv:PrinciplesCalculusForms}. In the last decades, tensor rank was studied in the case of tensors of order three in connection with the computational complexity of matrix multiplication \cite{Strassen:Gauss_elimination_not_optimal,Strassen:RanksOptimalComputation} and, more recently, in the higher order setting, in connection with the circuit complexity of certain families of polynomials \cite{Raz:TensorRankArithmeticFormulas}. In quantum information theory, tensor rank is used as a measure of entanglement in a quantum system \cite{YuChiGuoDu:TensorRankTripartiteW,DurVidCir:TrheeCubitsEntTwoIneqWays}. The notion of border rank is more geometric as it corresponds to membership into secant varieties of Segre varieties, objects that have been studied in algebraic geometry since the early twentieth century \cite{Terr:seganti}. It is known that asymptotic behaviors of tensor rank and tensor border rank of a given tensor are equivalent. In particular, upper bounds on border rank can be converted into upper bounds on rank which hold asymptotically \cite{Bini:RelationsExactApproxBilAlg}. We refer to \cite{Lan:TensorBook,BlaserNotes} for more information on the geometry of tensor spaces and their applications.

A natural question regarding tensor rank and border rank concerns their additivity properties under direct sum. Given $T \in V_1 \ootimes V_k$ and $S \in W_1 \ootimes W_k$, let $T \oplus S$ denote their direct sum, which is a tensor in $(V_1 \oplus W_1) \ootimes (V_k \oplus W_k)$. Subadditivity of tensor rank 
\[
\rmR(T \oplus S) \leq \rmR(T) + \rmR(S) 
\]
and border rank 
\[
\uR(T \oplus S) \leq \uR(T) + \uR(S) 
\]
follows directly from the definitions. It is natural to ask whether equality holds.

For $k =3$, examples where the inequality for border rank is strict were given by Sch\"onhage in \cite{Schon:PartTotalMaMu}: this construction is reviewed in Section \ref{subsec: schon}; briefly, for every $m,n \geq 1$, Sch\"onhage provided two tensors,
\[
\begin{array}{ll}
 T \in \bbC^{m+1} \otimes \bbC^{n+1} \otimes \bbC^{(m+1)(n+1)} & \text{with } \uR(T) = (m+1)(n+1), \\
 S \in \bbC^{nm} \otimes \bbC^{nm} \otimes \bbC^1 & \text{with } \uR(S) = mn,
\end{array}
 \]
 where $\uR(T \oplus S) = (m+1)(n+1)+1$. In particular, whenever either $m \geq 2$ or $n\geq 2$, one obtains an example of strict subadditivity.
 
The additivity problem for tensor rank of third order tensors was the subject of Strassen's additivity conjecture \cite{Str:VermeidungDiv}. This conjecture stated that tensor rank additivity under direct sum always holds. A great deal of work was devoted to this problem (see, e.g., \cite{FeiWin:DirectSumConj,JaTa:ValidDirectSumConjecture,CarCatChi:ProgressSymStrassen,Teit:SuffCondStrassen}) until 2017 when Shitov gave a counterexample \cite{Shitov:StrassenCounterexample}.

A tensor of order three can be regarded as a tensor of higher order by tensoring it with a tensor product of single vectors. For instance, a tensor $T \in V_1 \otimes V_2 \otimes V_3$ can be identified with a tensor of order four $T' = T \otimes e_0 \in V_1 \ootimes V_4$, where $V_4 = \langle e_0 \rangle$ is a one-dimensional space. Na\"ively, one would expect that Sch\"onhage's and Shitov's examples generalize to higher order settings via this identification. This is not the case, and intuitively the reason is that if $T' = T \otimes e_0$ and $S' = S \otimes e_0$, then $T' \oplus S' \neq (T\oplus S) \otimes e_0$. 

The problem of nonadditivity for rank and border rank of \emph{higher order} tensors is therefore open to our knowledge.

In this work, we settle the question for the case of border rank by providing examples of strict subadditivity for tensors of order four and higher. Our constructions are largely inspired by Sch\"onhage's. 

Sch\"onhage constructed his examples in order to provide new upper bounds on the asymptotic rank of the matrix multiplication tensor and thereby upper bounds on the exponent of matrix multiplication. We review this construction in Section \ref{subsec: schon}. The two key elements are the strictly subadditive upper bound $\uR(T \oplus S) < \uR(T) + \uR(S)$ and the fact that the Kronecker product $T \boxtimes S$ is a matrix multiplication tensor. Using these two facts, Sch\"onhage determined an upper bound on the direct sum of copies of the matrix multiplication tensor, exploiting the binomial expansion of $(T \oplus S)^{\boxtimes N}$ and the upper bound on its border rank. Strict subadditivity of tensors can therefore deliver nontrivial exponent bounds. At the time, this strategy gave the best bounds for the exponent of matrix multiplication and provided a sandbox example of Strassen's laser method, which is the technique used to obtain all subsequent upper bounds on the exponent \cite{Str:RelativeBilComplMatMult,CopperWinog:MatrixMultiplicationArithmeticProgressions,Stot:ComplexityMaMu,Wil:FasterThanCW,LeGall:PowersTensorsFMM,AlmWil:RefinedLaserMethodFMM}.

In our setting, the tensors $T \boxtimes S$ will be higher order generalizations of the matrix multiplication tensor. Some of these tensors were considered in \cite{ChrZui:TensorSurgery, ChrVraZui:AsyRankGraph}, and our work provides a new approach to the study of their exponents. The bounds presented here do not improve the best known upper bounds on the exponent of these tensors. However, the new technique provides nontrivial upper bounds and the strategies presented in this paper provide new and different types of tensor decompositions that are in many ways simpler or more direct when compared to the ones providing better bounds.

The results of this work hold over arbitrary fields as long as the characteristic is ``large enough''. We will not enter into details and we will work over the complex numbers for simplicity. We refer to \cite[Sec. 15.4]{BuClSho:Alg_compl_theory} for the formal definition of border rank and the details to extend the results over arbitrary fields.

The article is structured as follows. In Section \ref{section: preliminaries}, we provide mathematical preliminaries to our study as well as a review of Sch{\"o}nhage's construction. The new examples of strict subadditivity of border rank are presented in Section \ref{section: constructions}. The consequences on the asymptotic rank of generalizations of the matrix multiplication tensor are  presented in Section \ref{section: exponent}.

\subsection*{Acknowledgements} [M.~C. and F.~G.] This work was supported by VILLUM FONDEN via the QMATH Centre of Excellence (Grant No. 10059) and the European Research Council (Grant No. 818761). [J.~Z.] This material is based upon work directly supported by the National Science Foundation Grant No. DMS-1638352 and indirectly supported by the National Science Foundation Grant No. CCF-1900460. Any opinions, findings and conclusions or recommendations expressed in this material are those of the author(s) and do not necessarily reflect the views of the National Science Foundation.

\section{Preliminaries}\label{section: preliminaries}

In this section we discuss basic notions that will be used throughout the paper.

\subsection{Flattening maps of tensors and their image}
Every tensor naturally defines a collection of linear maps, called flattening maps. We will discuss here a characterization of tensor rank and border rank in terms of the image of a flattening map.

Let $T \in V_1 \ootimes V_k$ be a tensor of order $k$. The tensor~$T$ naturally induces a linear map 
\[
T: V_j^* \to V_1 \ootimes V_{j-1}\otimes V_{j+1} \ootimes V_k
\]
for every~$j = 1 \vvirg k$. We call these linear maps the \emph{flattening maps of $T$}. We say that~$T$ is \emph{concise} if all its flattening maps are injective. Each of the flattening maps uniquely determines~$T$. In fact, the image of any of them, say $T(V_k^*) \subseteq V_1 \ootimes V_{k-1}$, already uniquely determines $T$ up to the natural action of the general linear group $\GL(V_k)$. 

The following is a characterization of tensor rank and border rank via the geometry of the subspace $T(V_k^*)$. We refer to \cite[Theorem~2.5]{BucLan:RanksTensorsAndGeneralization} and \cite[Lemma~2.4]{GesOneVen:PartiallySymComon} for the proof and additional information.

\begin{proposition}\label{prop: rank and border rank via image of flattening}
 Let $T \in V_1 \ootimes V_k$ be a tensor. Let $E = T(V_k^*)\subseteq V_1 \ootimes V_{k-1}$ be the image of the last flattening map. Then
 \begin{align*}
  \rmR(T) &= \min \left\{r : E \subseteq \langle Z_1 \vvirg Z_r \rangle,\,\text{lin.~indep.\ } Z_i \in  V_1 \ootimes V_{k-1},\,\rmR(Z_i) = 1 \right\}\\
  \uR(T) &= \min\{r : E \subseteq \lim_{\eps \to 0} \langle Z_1(\eps) \vvirg Z_r(\eps) \rangle,\, \text{lin.~indep.\ }Z_i(\eps) \in  V_1 \ootimes V_{k-1},\, \rmR(Z_i(\eps)) = 1\},
   \end{align*}
   where the limit is taken in the Grassmannian of $r$-planes in $ V_1 \ootimes V_{k-1}$.
\end{proposition}

\begin{example}
 Consider the tensor $T = e_0\otimes e_0 \otimes e_1 + e_0\otimes e_1 \otimes e_0 + e_1\otimes e_0 \otimes e_0 \in \bbC^2 \otimes \bbC^2 \otimes \bbC^2$. It is known that $\rmR(T) = 3$ and $\uR(T) = 2$. Since $T$ is symmetric, the three flattening maps are equal. We have $T({\bbC^2}^*) = \langle e_1 \otimes e_0 + e_0 \otimes e_1, e_0 \otimes e_0\rangle \subseteq \bbC^2 \otimes \bbC^2$. The rank upper bound is immediate since $T({\bbC^2}^*) \subseteq \langle e_0 \otimes e_1,e_1 \otimes e_0 , e_0 \otimes e_0 \rangle$ showing $\rmR(T) \leq 3$. If $\rmR(T) \leq 2$, then $T({\bbC^2}^*)$ is spanned by two rank-one elements of $\bbC^2 \otimes \bbC^2$, but $T({\bbC^2}^*)$ only contains one rank-one element, up to scaling. This shows that $\rmR(T) = 3$. The border rank lower bound follows from the flattening lower bound: the border rank of~$T$ is at least the rank of any of the flattening maps $T : \bbC^2 \to \bbC^2 \otimes \bbC^2$, each of which equals 2. As for the border rank upper bound, let $E_\eps = \langle e_0 ^{\otimes 2}, (e_0 + \eps e_1)^{\otimes 2} \rangle$ and let $E_0 = \lim _{\eps \to 0} E_\eps$. Note that $E_0 = T({\bbC^2}^*)$. Indeed $e_0 \otimes e_0 \in E_\eps$ for every $\eps$, therefore $e_0 \otimes e_0 \in E_0$ as well. Moreover, $\frac{1}{\eps} [ (e_0 +\eps e_1)^{\otimes 2} - e_0^{\otimes 2} ] = e_0 \otimes e_1 + e_1 \otimes e_0 + \eps e_1 \otimes e_1 \in E_\eps $ for every $\eps$, so its limit as $\eps \to 0$ is an element of $E_0$. This shows that $ e_0 \otimes e_1 + e_1 \otimes e_0 \in E_0$. Hence we have the inclusion~$E_0 \subseteq T({\bbC^2}^*)$, and equality follows by dimension reasons.
\end{example}

\subsection{Degeneration, unit tensor and Kronecker product}\label{subsec:degen}
We now discuss a relation on tensors called degeneration and its connection to border rank and the asymptotic version of tensor rank.

The product group~$G = \GL(V_1) \ttimes \GL(V_k)$ naturally acts on the tensor space $V_1 \ootimes V_k$. Given two tensors $T,S \in V_1 \ootimes V_k$, we say that \emph{$S$ is a degeneration of $T$}, and write $S \degleq T$,~if 
\[
 S \in \bar{ G \cdot T }
\]
that is, $S$ belongs to the closure (equivalently in the Zariski or Euclidean topology) of the $G$-orbit of $T$. By re-embedding vector spaces in a larger common space, we may always assume that our tensors belong to the same space $V_1 \otimes \cdots \otimes V_k$. We will often tacitly identify tensors that are in the same $G$-orbit.

The notion of an identity matrix extends to $k$-tensors as follows. For $r \in \bbN$, let $V_j = \bbC^r$ and define the $k$-tensor
\[
\bfu_k(r) \coloneqq \sum_{i=1}^r e_i^{(1)} \ootimes e_i^{(k)} \in V_1 \ootimes V_k,
\]
where $e^{(j)}_1, \ldots, e^{(j)}_r$ is a fixed basis of $V_j$. The tensor~$\bfu_k(r)$ is sometimes called the rank-$r$ \emph{unit tensor}.

The fundamental relation between degeneration, unit tensors and border rank is that, for every $k$-tensor $T$ we have 
\begin{equation}\label{eqn: border rank deg}
 \uR(T) \leq r \text{ if and only if } T \degleq \bfu_k(r). 
\end{equation}

The \emph{Kronecker product} of two $k$-tensors $T \in V_1 \ootimes V_k$ and $S \in W_1 \ootimes W_k$ is the tensor $T \boxtimes S \in (V_1 \otimes W_1) \ootimes (V_k \otimes W_k)$ obtained from $T \otimes S \in V_1 \ootimes V_k \otimes W_1 \ootimes W_k$ by grouping together the spaces $V_j$ and $W_j$ for each $j$. Tensor rank and border rank are submultiplicative under the Kronecker product, that is, we have $\rmR(T \boxtimes S) \leq \rmR(T) \rmR(S)$ and $\uR(T \boxtimes S) \leq \uR(T) \uR(S)$. Both inequalities may be strict.

In the context of the study of the arithmetic complexity of matrix multiplication, Strassen introduced an asymptotic notion of tensors rank~\cite{Str:AsySpectrumTensors}, called asymptotic rank, and developed the theory of asymptotic spectra of tensors to gain a deep understanding of its properties \cite{Str:AsySpectrumTensorsExpMatMult,Str:DegComplBilMapsSomeAsySpec} (see also \cite{ChrVraZui:UniversalPtsAsySpecTensors}). The \emph{asymptotic rank} of $T \in V_1 \otimes \cdots \otimes V_k$ is defined as
\[
 \aR(T) = \lim_{N \to \infty} ( \rmR(T^{\boxtimes N}) ) ^{1/N}.
\]
It will often be convenient to take the logarithm of the asymptotic rank,
\[
\omega(T) \coloneqq \log (\aR(T)),
\]
which is called the \emph{exponent} of $T$. We write $\log \coloneqq \log_2$, the logarithm in base $2$. The limit in the definition of asymptotic rank exists by Fekete's Lemma (see, e.g., \cite[page~189]{PolSze:ProblemsTheoremsAnalysisI}), via submultiplicativity of tensor rank. The notion of asymptotic rank does not depend on whether one uses tensor rank $\rmR(T)$ or border rank $\uR(T)$ in the definition \cite{Bini:RelationsExactApproxBilAlg, Str:RelativeBilComplMatMult}. Because of the submultiplicative property of tensor rank and border rank, we have that $\aR(T) \leq \uR(T) \leq \rmR(T)$. 

The importance of asymptotic rank in the study of the arithmetic complexity of matrix multiplication comes from the following connection (we refer to \cite{BlaserNotes} for more information).
For $m_1, m_2, m_3 \in \bbN$ the \emph{matrix multiplication tensor} $\Mamu(m_1,m_2,m_3)$ is defined as
\begin{equation}\label{eqn: mamu tensor}
\Mamu(m_1,m_2,m_3) \coloneqq \sum_{i_1 = 1}^{m_1} \sum_{i_2 = 1}^{m_2} \sum_{i_3 = 1}^{m_3} e_{i_1, i_2} \otimes e_{i_2, i_3} \otimes e_{i_3, i_1} \in \bbC^{m_1m_2} \otimes \bbC^{m_2 m_3} \otimes \bbC^{m_3m_1}.
\end{equation}
This tensor defines the bilinear map $\bbC^{m_1m_2} \times  \bbC^{m_2 m_3} \to \bbC^{m_3m_1}$ which multiplies a matrix of size $m_1 \times m_2$ with one of size $m_2 \times m_3$. It is a fundamental result that the tensor rank of $\Mamu(m_1,m_2,m_3)$ characterizes the arithmetic complexity (i.e., the minimal number of  scalar additions and multiplications in any arithmetic algorithm) of matrix multiplication. In particular, for every $\varepsilon > 0$ the arithmetic complexity of $n\times n$ matrix multiplication is $\mathcal{O}(n^{\omega + \varepsilon})$ where $\omega = \omega(\Mamu(2,2,2))$. It is a major open problem whether $\omega$ equals~$2$ or is strictly larger than~$2$ \cite{BuClSho:Alg_compl_theory}.

The notion of the exponent of a tensor naturally extends to a relation on tensors called \emph{relative exponent} or \emph{rate of asymptotic conversion} \cite[Definition 1.7]{ChrVraZui:AsyRankGraph}. Following that terminology, the exponent of a $k$-tensor $T$ equals the asymptotic rate of conversion from the unit tensor~$\bfu_k (2)$ to $T$.

\subsection{Graph tensors} 

Graph tensors are a natural generalization of matrix multiplication tensors.  They are defined as a Kronecker product of unit tensors of lower order according to the structure of a hypergraph \cite{ChrZui:TensorSurgery}.

Let $G$ be a hypergraph with vertex set $V(G) = \{ 1\vvirg k\}$ and edge set~$E(G)$, that is, $E(G)$ is a set of subsets of $V(G)$. For every hyperedge $I \in E(G)$, let $ n_I \in \bbN $ be integer weight.

For every hyperedge $I = \{ i_1 \vvirg i_p\}$, define the $k$-tensor 
\[
 \bfu_{(I)}(n_I) := \Bigl[ \sum_{j = 1}^{n_I} e^{(i_1)}_{j} \ootimes e^{(i_p)}_j \Bigr] \otimes \Bigl[\bigotimes_{i' \notin I}\, e^{(i')}_{0}\Bigr] \in \Bigl(\bigotimes_{i\in I} \bbC^{n_I}\Bigr) \otimes \Bigl(\bigotimes_{i' \notin I} \bbC^1\Bigr),
\]
where $e^{(i)}_{1} \vvirg e^{(i)}_{n_I}$ is a fixed basis of $\bbC^{n_I}$ for every $i \in I$, and $e^{(i')}_0$ is a fixed basis element of $\bbC^1$ for $i' \notin I$.

The \emph{graph tensor} associated to the hypergraph $G$ with weights $\bfn = (n_I : I \in E(G))$ is defined as
\[
 T(G,\bfn) \coloneqq \bigboxtimes_{I \in E(G)} \ \bfu_{(I)}(n_I),
\]
where $\boxtimes$ denotes the Kronecker product. Thus $T(G,\bfn)$ is a $k$-tensor in $V_1 \otimes \cdots \otimes V_k$ whose $j$-th factor has a local structure $V_j = (\bigotimes _{I \ni j} \bbC^{n_I}) \otimes (\bigotimes_{I \not\ni j} \bbC^1)$. In particular, $\dim V_j = \prod_{I \ni j} \dim n_I$.

In the language of tensor networks, $T(G)$ is the \emph{generic tensor} in the tensor network variety associated to the graph $G$, as long as the local dimensions are at least as large as $\dim V_j$, see e.g. \cite[Ch. 12]{Hackbusch:TensorBook}, \cite{LanQiYe:GeomTensorNetwork}.

An important feature of graph tensors is their \emph{self-reproducing} property: if $G$ is a hypergraph with weights $\bfn = ( n_I : I \in E(G))$ and $T = T(G , \bfn)$ is the associated graph tensor, then $T^{\boxtimes N} = T(G, \bfn^{\odot N})$ where $\bfn^{\odot N}$ is the tuple of weights obtained from $\bfn$ by raising every entry to the $N$-th power. 

\begin{example}
Let $G = K_3$ be the triangle graph, that is, $G$ has vertex set $V(G) = \{1,2,3\}$ and edge set $E(G) = \{\{1,2\},\{2,3\},\{3,1\}\}$ which we write shortly as $E(G) = \{12, 23, 31\}$. Consider weights on $G$ given by $\bfn = (n_{12},n_{23},n_{31})$. The graph tensor associated to $G$ is the tensor $T(G, \bfn) \in V_1 \otimes V_2 \otimes V_3$ with $V_1 = \bbC^{n_{31}} \otimes \bbC^{n_{12}}$, $V_2 = \bbC^{n_{12}} \otimes \bbC^{n_{23}}$ and $V_3 = \bbC^{n_{23}} \otimes \bbC^{n_{31}}$ given by
\[
 T(G ,\bfn) = \sum e_{i_{31}i_{12}} \otimes e_{i_{12}i_{23}} \otimes e_{i_{23}i_{31}},
\]
where the sum ranges over the indices $i_{12}, i_{23}, i_{31}$ with $i_{12} = 1 \vvirg n_{12}$ and similarly for $i_{23}, i_{31}$. Thus $T(G,\bfn)$ equals the matrix multiplication tensor $\Mamu(n_{12},n_{23},n_{31})$ in \eqref{eqn: mamu tensor}. In general, we may represent any graph tensor $T(G ,\bfn)$ by the defining weighted graph with vertices labeled by the appropriate vector spaces $V_i$. In this case,
\[
T(G,\bfn) \quad =  \quad \begin{minipage}{.3\textwidth}       \begin{tikzpicture}[scale=1]
\draw (0,1.5)-- (0,0);
\draw (0,0)-- (1.5,0);
\draw (0,1.5)-- (1.5,0);
\draw (.75,0) node[anchor=north] {$n_{12}$};
\draw (0.75,0.75) node[anchor=south west] {$n_{23}$};
\draw (0,.75) node[anchor=east] {$n_{31}$};
\draw[fill=black] (0,0) circle (.1cm);
\draw[fill=black] (1.5,0) circle (.1cm);
\draw[fill=black] (0,1.5) circle (.1cm);
\draw (0,1.6) node[anchor=south] {$V_3$};
\draw (1.6,0) node[anchor=west] {$V_2$};
\draw (0,0) node[anchor=north east] {$V_1$};
\end{tikzpicture}
\end{minipage}.
\]
We will often drop the notation $V_i$ from the picture. 

More generally, the graph tensor associated to the cycle graph $C_k$ of length $k$ is the iterated matrix multiplication tensor of order $k$.
\end{example}

\begin{example}[Unit tensors]
For any $k$ let $G$ be the graph with vertex set $V(G) = \{1 \vvirg k\}$ and edge set $E(G) = \{\{1, \ldots, k\}\}$. That is, $G$ has a single hyperedge containing all vertices. Consider the weight $\bfn = r \in \bbN$ for this hyperedge. Then the associated graph tensor $T(G,\bfn)$ equals the unit tensor~$\bfu_k(r)$ defined in Subsection~\ref{subsec:degen}. For the case $k=3$, the graphical representation for this graph tensor is:
 \[
T(G,r) \quad =  \quad \begin{minipage}{.3\textwidth}       \begin{tikzpicture}[scale=1]
\draw [rounded corners=2mm,fill=gray!50] (-.2,-.2)--(2,-.2)--(-.2,2)--cycle;
\draw [rounded corners=2mm,fill=white] (.05,.05)--(1.4,.05)--(.05,1.4)--cycle;
\draw[fill=black] (0,0) circle (.1cm);
\draw[fill=black] (1.5,0) circle (.1cm);
\draw[fill=black] (0,1.5) circle (.1cm);
\draw (.85,.85) node[anchor=south west] {$r$};
\draw (.2,1.6) node[anchor=south] {$V_3$};
\draw (1.6,.2) node[anchor=west] {$V_2$};
\draw (-.1,-.1) node[anchor=north east] {$V_1$};
\end{tikzpicture}
\end{minipage}
\]
  \end{example}

Back to the general setting, since border rank is submultiplicative under the Kronecker product, we have a trivial upper bound for the asymptotic rank of graph tensors given by the product of the rank of the factors from which they arise. In particular, we have the asymptotic rank upper bound
\begin{equation}\label{eqn: trivial asy rank}
 \aR(T(G,\bfn)) \leq \uR(T(G,\bfn)) \leq \prod_{\mathclap{I \in E(G)}} n_I.
  \end{equation}
Consequently, the exponent of $T(G,\bfn)$ is bounded from above by the logarithm of the right hand side of \eqref{eqn: trivial asy rank}, that is $\omega(T(G,\bfn)) \leq \sum_{I \in E(G)} \log(n_I)$.

\subsection{Sch\"onhage's construction and the exponent of matrix multiplication}\label{subsec: schon}

We review Sch\"onhage's construction of strict subadditivity of border rank of $3$-tensors under the direct sum. The higher order examples in Section~\ref{section: constructions} are largely inspired by this construction.

Fix $n_1,n_2 \geq 1$ and consider the two tensors associated to the following graphs:
\[
   T_1 \quad = \qquad \begin{minipage}{.3\textwidth}
       \begin{tikzpicture}[scale=1]
\draw (0,1.5)-- (0,0);
\draw (0,0)-- (1.5,0);
\draw (.75,0) node[anchor=north] {$n_1 +1$};
\draw (0,.75) node[anchor=east] {$n_2 +1$};
\draw (0,1.6) node[anchor=south] {$V_3$};
\draw (1.6,0) node[anchor=west] {$V_2$};
\draw (0,0) node[anchor=north east] {$V_1$};
\draw[fill=black] (0,1.5) circle (.1cm);
\draw[fill=black] (1.5,0) circle (.1cm);
\draw[fill=black] (0,0) circle (.1cm);
\end{tikzpicture}
      \end{minipage}
      \qquad
        T_2 \quad = \qquad \begin{minipage}{.3\textwidth}
       \begin{tikzpicture}[scale=1]
\draw (0,1.5)-- (1.5,0);
\draw (0,1.6) node[anchor=south] {$W_3$};
\draw (1.6,0) node[anchor=west] {$W_2$};
\draw (0,0) node[anchor=north east] {$W_1$};
\draw (.7,.7) node[anchor=south west] {$n_1n_2$};
\draw[fill=black] (0,1.5) circle (.1cm);
\draw[fill=black] (1.5,0) circle (.1cm);
\draw[fill=black] (0,0) circle (.1cm);
\end{tikzpicture}
      \end{minipage}
\]

It is immediate that $\uR(T_1) = (n_1+1)(n_2+1)$ and $\uR(T_2) = n_1n_2$, so that one obtains the trivial upper bound on the direct sum: $\uR(T_1 \oplus T_2) \leq (n_1+1)(n_2+1) + n_1n_2$. Sch\"onhage proved $\uR(T_1 \oplus T_2) = (n_1+1) ( n_2+1) + 1$~\cite{Schon:PartTotalMaMu} (see also \cite{BlaserNotes}). In particular, whenever $n_1 \geq 2$ or $n_2 \geq 2$, this construction provides an example of strict subadditivity of border rank. 

Note that $T_1 \boxtimes T_2$ is the matrix multiplication tensor with edge weights $\bfn = (n_1+1,n_2+1, n_1n_2)$. Using the strict subadditivity result Sch\"onhage provided an upper bound on the exponent of matrix multiplication. We provide two key results which are useful to reproduce Sch\"onhage's upper bound on the exponent of matrix multiplication as well as the upper bounds on the exponent of certain graph tensors in Section \ref{section: exponent}. We refer to \cite{BlaserNotes} and \cite[Sec. 2]{Zuid:Thesis} for additional information.

\begin{lemma}\label{catal}
Let $S,T, U$ be tensors such that $S \boxtimes T \degleq S \boxtimes U$. Then for every $N \in \bbN$ we have
\[
S \boxtimes T^{\boxtimes N}  \degleq  S \boxtimes U^{\boxtimes N}.
\]
In particular, if $\bfu_k(s) \boxtimes T \degleq \bfu_k( r )$ for some integers $r,s$, then for all $N \in \bbN$ we have
\[
\bfu_k(s) \boxtimes T^{\boxtimes N} \degleq \bfu_k(s) \boxtimes \bfu_k( \bigl\lceil \tfrac{r}{s} \bigr\rceil^N ).
\]
\end{lemma}
\begin{proof}
 The proof is by induction. The base case $S \boxtimes T \degleq  S \boxtimes U$ is true by assumption. The induction step is
\[
S \boxtimes T^{\boxtimes n} = S \boxtimes T \boxtimes  T^{\boxtimes  (n- 1)} \degleq  S \boxtimes U \boxtimes T^{\boxtimes (n-1)} \degleq  S \boxtimes U \boxtimes U^{\boxtimes (n-1)} = S \boxtimes U^{\boxtimes n},
\]
where we first use the assumption in the inequality $S \boxtimes T \degleq  S \boxtimes U$ and then we use the inductive hypothsis in the inequality $S \boxtimes T^{\boxtimes( n-1)} \degleq  S \boxtimes U^{\boxtimes (n-1)}$. 

If $\bfu_k(s) \boxtimes T \degleq \bfu_k(r)$, then $\bfu_k(s) \boxtimes T  \degleq \bfu_k(s) \boxtimes \bfu_k( \lceil r/s\rceil )$. Applying the first part of the Lemma with $S = \bfu_k(s)$ and $U =  \bfu_k( \lceil r/s \rceil )$ provides the desired result.
\end{proof}

\begin{proposition}\label{prop: bound exponent general}
Let $T_1 \in V_1 \ootimes V_k$ and $T_2 \in W_1 \ootimes W_k$ be two tensors. Suppose $\uR(T_1 \oplus T_2) \leq r$. Let $N \geq 0$ be an integer and let $p \in (0,1)$ such that $pN$ is an integer. Then
\[
\uR(  T_1^{\boxtimes Np} \boxtimes T_2^{\boxtimes N(1-p)}) \leq \left(\frac{r}{2^{h(p)+o(1)}} \right) ^N
\]
where $h(p)$ is the binary entropy function $h(p) = - p \log(p) - (1-p) \log(1-p)$.
\end{proposition}
\begin{proof}
Consider the binomial expansion of $(T_1 \oplus T_2) ^{\boxtimes N}$:
\[
 (T_1 \oplus T_2) ^{\boxtimes N} = \bigoplus_{M = 0}^N \bfu_k \Bigl({\textbinom{N}{M}} \Bigr) \boxtimes (T_1 ^{\boxtimes M} \boxtimes T_2^{\boxtimes(N-M)}).
\]
It is immediate that the right-hand side above degenerates to each direct summand: in particular $ (T_1 \oplus T_2) ^{\boxtimes N} \deggeq \binom{N}{pN}\boxtimes (T_1 ^{\boxtimes pN} \boxtimes T_2^{\boxtimes(1-p)N})$. 

Moreover, since $\uR(T_1 \oplus T_2) \leq r$, from \eqref{eqn: border rank deg}, we obtain $T_1 \oplus T_2 \degleq \bfu_k(r)$, and therefore $(T_1 \oplus T_2)^{\boxtimes N} \degleq \bfu_k(r)^N$. Thus, 
\[
\bfu_k (r^N) \deggeq \bfu_k \left({\textbinom{N}{pN}} \right) \boxtimes (T_1 ^{\boxtimes pN} \boxtimes T_2^{\boxtimes((1-p)N)}) .
\]
Using Lemma \ref{catal}, we have
\[
 \bfu_k \left( r^N /  \left({\textbinom{N}{pN}} \right) \right) \deggeq \bfu_k \left({\textbinom{N}{pN}} \right) \boxtimes (T_1 ^{\boxtimes pN} \boxtimes T_2^{\boxtimes((1-p)N)}) \deggeq T_1 ^{\boxtimes pN} \boxtimes T_2^{\boxtimes((1-p)N)}
\]
Recall that $\textbinom{N}{pN} = 2^{N h(p) + o(1)}$ where $h(p)$ is the binary entropy function. This gives
\[
 \uR( T_1 ^{\boxtimes pN} \boxtimes T_2^{\boxtimes((1-p)N)} )  \leq \left(\frac{r}{2^{h(p)+o(1)}} \right) ^N,
\]
and concludes the proof.
\end{proof}

Because of the self-reproducing property of graph tensors, it is convenient to allow the weights of the graph to have fractional exponents. We will use this convention in order to give asymptotic statements with the understanding that the statement holds for the Kronecker powers for which the dimensions have integer values. More precisely, given a tensor $T$ and values $q \in (0,1)$ and $\rho \geq 0$, the statement $\aR(T^{\boxtimes q}) \leq \rho$ is to be read as $\uR(T^{\boxtimes Nq}) \leq \rho^{N + o(1)}$ for all $N$ for which $qN$ is an integer. From this point of view, after taking an $N$-th root in Proposition \ref{prop: bound exponent general}, we obtain the asymptotic bound
\[
\aR\left( T_1 ^{\boxtimes p} \boxtimes T_2^{\boxtimes(1-p)} \right) \leq  \frac{r}{2^{h(p)}} .
\]
After taking the logarithm, we have a bound on the exponent
\begin{equation}\label{eqn: bound exponent general}
\omega \left(T_1 ^{\boxtimes p} \boxtimes T_2^{\boxtimes(1-p)}\right) \leq  \log(r) - h(p) .
 \end{equation}
Sch\"onhage's construction provides tensors $T_1,T_2$ with $\uR(T_1 \oplus T_2) = (n_1+1)(n_2+1)+1$ and $T_1^{p} \boxtimes T_2^{1-p} = \Mamu((n_1+1)^p,(n_2+1)^p, (n_1n_2)^{1-p})$. Applying Proposition \ref{prop: bound exponent general}, one obtains
\[
 \omega\left(\Mamu((n_1+1)^p,(n_2+1)^p, (n_1n_2)^{1-p}) \right) \leq \log( (n_1+1)(n_2+1)+1 )- h(p).
\]
For $n_1=n_2=3$, we obtain $\omega ( \Mamu(4^p,4^p,9^{1-p})) \leq \log(17) - h(p)$. Cyclically permuting the factors and using the self-reproducing property of the matrix multiplication tensor, one obtains an upper bound on the exponent of a square matrix multiplication and, by passing to the asymptotic rank,
\[
 \omega(\Mamu(2,2,2)) \leq \frac{3(\log(17) - h(p))}{ 4p + (1-p) \log(9)}.
\]
The right hand side attains its minimum at $p \approx 0.61$, giving Sch\"onhage's upper bound on the exponent $\omega( \Mamu(2,2,2)) \leq 2.55$.

\section{Strict subadditivity of border rank}\label{section: constructions}

In this section we provide four families of examples of strict subadditivity of border rank for higher order tensors. The subadditivity results are recorded in Theorem \ref{thm: matrix half size}, Theorem \ref{thm: absorb big leg}, Theorem \ref{thm: matrix on many legs} and Theorem \ref{thm: absorbing GHZ}.

All constructions are characterized by a structure similar to Sch\"onhage's. We consider two graph tensors: 
\begin{itemize}
 \item[$\cdot$] The tensor $T_1$ is a \emph{spider}, that is, a graph tensor where the underlying graph has all edges incident to a single vertex. In this case, the graph tensor is, up to change of coordinates, the only concise tensor in its space.
 \item[$\cdot$] The tensor $T_2$ is either a matrix, that is, a graph tensor with a single edge, or $\bfu_3(r)$, that is, a graph tensor with a single hyperedge of order three.
\end{itemize}
Constructions 1, 2 and 3 add a matrix to the spider. Construction 1 provides a construction for tensors of order $4$ where the direct sum attains minimal border rank. For large edge dimensions, the border rank upper bound is roughly $2/3$ times the trivial additive upper bound. Construction 2 provides an improvement of Construction 1 for certain smaller edge dimensions. Construction 3 concerns tensors of all orders and gives an optimal savings of a factor of $2$ for large edge dimensions. Construction 4 adds a unit tensor to the legs of a three-legged spider. 

\subsection*{Construction 1: Adding a matrix}

This first construction concerns tensors of order four. Fix $n_1,n_2,n_3 \geq 2$ with $n_1$ (or $n_2$ or $n_3$) odd. Consider the following two tensors:
\[
   T_1 \quad = \qquad \begin{minipage}{.3\textwidth}
       \begin{tikzpicture}[scale=1]
\draw (0,1.5)-- (0,0);
\draw (0,0)-- (1.5,0);
\draw (0,0)-- (-1,-1);
\draw[fill=black] (1.5,0) circle (.1cm);
\draw[fill=black] (0,1.5) circle (.1cm);
\draw[fill=black] (0,0) circle (.1cm);
\draw[fill=black] (-1,-1) circle (.1cm);
\draw (.75,0) node[anchor=north] {$n_1 +1$};
\draw (0,.75) node[anchor=east] {$n_2 +1$};
\draw (-.5,-.5) node[anchor=south east] {$n_3 +1$};
\draw (-1,-1) node[anchor=north east] {$V_3$};
\draw (1.55,0) node[anchor=west] {$V_1$};
\draw (0,1.55) node[anchor=south] {$V_2$};
\draw (0,-.1) node[anchor=south east] {$V_4$};
\end{tikzpicture}
      \end{minipage}
      \qquad
        T_2 \quad = \qquad \begin{minipage}{.3\textwidth}
       \begin{tikzpicture}[scale=1]
\draw (0,1.5)-- (1.5,0);
\draw[fill=black] (1.5,0) circle (.1cm);
\draw[fill=black] (0,1.5) circle (.1cm);
\draw[fill=black] (0,0) circle (.1cm);
\draw[fill=black] (-1,-1) circle (.1cm);
\draw (.7,.7) node[anchor=south west] {$N$};
\draw (-1,-1) node[anchor=north east] {$W_3$};
\draw (1.55,0) node[anchor=west] {$W_1$};
\draw (0,1.55) node[anchor=south] {$W_2$};
\draw (0,-.1) node[anchor=south east] {$W_4$};
\end{tikzpicture}
      \end{minipage}
\]
where $N = \frac{1}{2}(n_1-1)(n_2-1)(n_3-1)$. In this case, we have the following result.

\begin{theorem}\label{thm: matrix half size}
For every $n_1,n_2,n_3$ with $n_1$ odd, we have 
\begin{align*}
\uR(T_1)&= (n_1+1)(n_2+1)(n_3+1),\\ 
\uR(T_2)&= N,
\end{align*}
and
\[
\uR(T_1 \oplus T_2) = (n_1+1)(n_2+1)(n_3+1) + 1. 
\]
\end{theorem}

\begin{proof}
For $p=1,2,3$, write $V_p = \bbC^{n_p+1}$ and let $V_4 = \bbC^{(n_1+1)(n_2+1)(n_3+1)}$. Let $\{ v^p_j : j =0 \vvirg n_p\}$ be a basis of $V_p$ and $\{ v^{4}_{i_1,i_2,i_3} : i_p= 0 \vvirg n_p\}$ be a basis of $V_4$. We have $T_1 \in V_1 \ootimes V_4$.

Similarly, for $p = 1,2$, let $W_p = \bbC^N$ and for $p=3,4$ let $W_p = \bbC^1$. Write $m_1 = \frac{1}{2}(n_1 - 1)$, $m_2 = n_2-1$ and $m_3 = n_3-1$. For $p=1,2$ let $\{w^p_{j_1,j_2,j_3} : j_p = 1 \vvirg m_p\}$ be a basis of $W_p$ and let $W_p = \langle w^p\rangle$ for $p =3,4$; note that indeed these are $\frac{n_1-1}{2} (n_2-1)(n_3-1) = N$ vectors. We have $T_2 \in W_1 \ootimes W_4$.

Regard $T_1 \oplus T_2$ as a tensor in $(V_1 \oplus W_1) \ootimes (V_4 \oplus W_4)$.

The values of $\uR(T_1)$ and $\uR(T_2)$ are immediate. The lower bound $\uR(T_1 \oplus T_2) \geq (n_1+1)(n_2+1)(n_3+1) + 1$ follows by conciseness.

For the upper bound, we determine a set of $(n_1+1)(n_2+1)(n_3+1) + 1$ rank-one elements $\calZ_\eps \subseteq (V_1 \oplus W_1)\otimes (V_2 \oplus W_2) \otimes (V_3 \oplus W_3)$ such that $(T_1 \oplus T_2)(V_4^* \oplus W_4^*) \subseteq \lim \langle \calZ_\eps \rangle$. By Proposition \ref{prop: rank and border rank via image of flattening}, this provides the desired upper bound.

Note
\[
 (T_1 \oplus T_2)(V_4^* \oplus W_4^*) = V_1 \otimes V_2 \otimes V_3 \oplus \langle \bfu(N) \rangle
\]
where 
\[
\bfu(N) := \sum_{\substack{ j_p = 1 \vvirg m_p \\ p=1,2,3} } w^1_{j_1, j_2,j_3} \otimes w^2_{j_1,j_2,j_3} \otimes w^3 = \bfu_{2}(N) \otimes w^3\in W_1 \otimes W_2 \otimes W_3.
\]

We will denote the elements of $\calZ_\eps$ using indices $\{ -1, (0,0,0) \vvirg (n_1,n_2,n_3)\}$; note that these are $(n_1+1)(n_2+1)(n_3+1) + 1$ elements. We drop the dependency on $\eps$ from the notation.

For $p=1 ,2,3$ and $j_p = 1 \vvirg m_p$, define
\[
 Z_{j_1,j_2,j_3}= (v^1_{j_1} + \eps w^1_{j_1,j_2,j_3} ) \otimes (v^2_{j_2} + \eps w^2_{j_1,j_2,j_3} ) \otimes (v^3_{j_3} + \eps w^3). 
\]
Write $\bfZ_1 = \sum _{j_1,j_2,j_3} Z_{j_1,j_2,j_3}$ for the tensor obtained as sum of the $m_1m_2m_3 = \frac{n_1-1}{2}(n_2-1)(n_3-1)$ rank-one tensors defined above. The component of degree $3$ (with respect to $\eps$) in $\bfZ_1$ is exactly $\bfu(N)$.

For $j_1 = 1 \vvirg m_1$ (so that $m_1+j_1 = m_1+1 \vvirg n_1-1)$, $j_2 = 1 \vvirg m_2$ and $j_3 = 1 \vvirg m_3$, define
\[
 Z_{m_1 + j_1,j_2,j_3} = (v^1_{m_1 + j_1} + \eps w^1_{j_1,j_2,j_3} ) \otimes (v^2_{j_2} - \eps w^2_{j_1,j_2,j_3}) \otimes v^3_{j_3}.
\]
Let $\bfZ_{110}$ be the sum of the tensors just defined.

For $k_1 = 1 \vvirg m_1$, and for $k_2 = 1 \vvirg m_2$ define the two sets of tensors
\begin{align*}
Z_{n_1,k_2,0} &= (v^1_{n_1} + \eps \sum_{\substack{p = 1,3 \\ j_p = 1 \vvirg m_p}} w^1_{j_1 ,k_2, j_3} ) \otimes v^2_{k_2} \otimes ( v^3_0 - \eps w^3), \\
 Z_{k_1,n_2,0} &= v^1_{k_1} \otimes ( v^2_{n_2} + \eps \sum_{\substack{p = 2,3 \\ j_p = 1 \vvirg m_p}} w^1_{k_1, j_2 ,j_3 } )  \otimes ( v^3_0 - \eps w^3) ,
\end{align*}
consisting, respectively, of $n_2-1$ and $\frac{n_1-1}{2}$ rank-one tensors. Write $\bfZ_{101}$ and $\bfZ_{011}$ for the sum of the first and second sets of tensors just defined.

Now, the component of degree $2$ in $\bfZ_1$ is opposite to the component of degree $2$ in $\bfZ_{110} + \bfZ_{101} + \bfZ_{011}$. Let $S = \bfZ_1 + \bfZ_{110} + \bfZ_{101} + \bfZ_{011}$. We deduce that the component of degree $2$ in $S$ is $0$.

Therefore $S$ can be written as $S = S_0 + \eps S_1 + \eps^3 \bfu(N)$ and 
\begin{align*}
 S_1 = & \sum_{\substack{i_1 =1 \vvirg n_1 \\ i_2 = 1\vvirg n_2}} v^1_{i_1} \otimes v^2_{i_2} \otimes \omega^3_{i_1,i_2} +  \sum_{\substack{i_1 =1 \vvirg n_1 \\ i_3 = 1\vvirg n_3}} v^1_{i_1} \otimes \omega^2_{i_1,i_3} \otimes v^3_{i_3} +  \sum_{\substack{i_2 =1 \vvirg n_2 \\ i_3 = 1\vvirg n_3}} \omega^1_{i_1,i_2} \otimes v^2_{i_2} \otimes v^3_{i_3} ,\\
\end{align*}
for some vectors $\omega^1_{i_2,i_3} \in W_1$, $\omega^2_{i_1,i_3} \in W_2$ and $\omega^3_{i_1,i_2} \in W_3$.

Define
\begin{align*}
Z_{0,i_2,i_3} = (v^1_0  - \eps \omega^1_{i_2,i_3}) \otimes v^2_{i_2} \otimes v^3_{i_3} ,\\
Z_{i_1,0,i_3} = v^1_{i_1}   \otimes (v^2_{0} - \eps \omega^2_{i_1,i_3})\otimes v^3_{i_3}, \\
Z_{i_1,i_2,0} = v^1_{i_1}  \otimes v^2_{i_2} \otimes (v^3_{0} - \eps \omega^3_{i_1,i_2}).\\
\end{align*}
Let $\bfZ_{0,0,0}$ be the sum of these three families of tensors. Then $S + \bfZ_{0,0,0} = R + \eps^3 \bfu(N)$ for some tensor $R$ not depending on $\eps$: in particular, if $\Phi \subseteq [0,n_1] \times [0,n_2] \times [0,n_3]$ is the subset of indices $(i_1,i_2,i_3)$ for which a tensor $Z_{i_1,i_2,i_3}$ has been defined, then $R = \sum_{(i_1,i_2,i_3) \in \Phi} Z_{i_1,i_2,i_3}|_{\eps = 0}$.

Let $\Omega \subseteq [0,n_1] \times [0,n_2] \times [0,n_3] $ be the set of all the triples $(i_1,i_2,i_3)$ for which a tensor $Z_{i_1,i_2,i_3}$ has not yet been defined; in other words $\Omega$ is the complement of $\Phi$. For $(i_1,i_2,i_3) \in \Omega$, let $Z_{i_1,i_2,i_3} = v^1_{i_1} \otimes v^2_{i_2} \otimes v^3_{i_3}$.

Finally, define $Z_{-1} = (\textsum_{0}^{n_1} v^1_{i_1} ) \otimes (\textsum_{0}^{n_2} v^2_{i_2} ) \otimes (\textsum_{0}^{n_3} v^3_{i_3} )$. Note 
\[
Z_{-1} = \sum_{(i_1,i_2,i_3) \in [0,n_1] \times [0,n_2] \times [0,n_3] } v^1_{i_1} \otimes v^2_{i_2} \otimes v^3_{i_3} 
\]
equals the sum over the indices of $\Phi$ and of $\Omega$.

Let $\calZ_\eps = \{ Z_{-1}, Z_{0,0,0} \vvirg Z_{n_1,n_2,n_3} \}$: then $\calZ_\eps$ has $(n_1+1)(n_2+1)(n_3+1)+1$ elements. Let $E_\eps = \langle \calZ_\eps \rangle \subseteq (V_1\oplus W_1) \otimes (V_2\oplus W_2) \otimes (V_3\oplus W_3)$ and $E_0 = \lim_{\eps \to 0} E_\eps$ where the limit is taken in the corresponding Grassmannian.

We show that $(T_1 \oplus T_2) (V_4^* \oplus W_4^*) \subseteq E_0$ (and in fact equality holds).

For every $(i_1,i_2,i_3)$, we have $Z_{i_1,i_2,i_3}|_{\eps = 0} = v^1_{i_1} \otimes v^2_{i_2} \otimes v^3_{i_3}$. This shows $V_1 \otimes V_2 \otimes V_3 \subseteq E_0$.

Moreover $\eps^3 \bfu(N) = \sum_{i_1,i_2,i_3} Z_{i_1,i_2,i_3} - Z_{-1}$, therefore $\bfu(N) \in E_\eps$ for every $\eps$, hence $\bfu(N) \in E_0$.

This shows $(T_1 \oplus T_2) (V_4^* \oplus W_4^*) \subseteq E_0$ and concludes the proof.
\end{proof}

\subsection*{Construction 2: Adding a matrix, II}

Construction 1 does not apply in the case where the weights of the edges are $2$. Construction 2 addresses this setting in a particular case. Fix $a \geq 2$. Consider the two tensors
\[
   T_1 \quad = \qquad \begin{minipage}{.3\textwidth}
       \begin{tikzpicture}[scale=1]
\draw (0,1.5)-- (0,0);
\draw (0,0)-- (1.5,0);
\draw (0,0)-- (-1,-1);
\draw[fill=black] (1.5,0) circle (.1cm);
\draw[fill=black] (0,1.5) circle (.1cm);
\draw[fill=black] (0,0) circle (.1cm);
\draw[fill=black] (-1,-1) circle (.1cm);
\draw (.75,0) node[anchor=north] {$2$};
\draw (0,.75) node[anchor=east] {$2$};
\draw (-.5,-.5) node[anchor=south east] {$a+2$};
\draw (-1,-1) node[anchor=north east] {$V_3$};
\draw (1.55,0) node[anchor=west] {$V_1$};
\draw (0,1.55) node[anchor=south] {$V_2$};
\draw (0,-.1) node[anchor=south east] {$V_4$};
\end{tikzpicture}
      \end{minipage}
      \qquad
        T_2 \quad = \qquad \begin{minipage}{.3\textwidth}
       \begin{tikzpicture}[scale=1]
\draw (0,1.5)-- (1.5,0);
\draw[fill=black] (1.5,0) circle (.1cm);
\draw[fill=black] (0,1.5) circle (.1cm);
\draw[fill=black] (0,0) circle (.1cm);
\draw[fill=black] (-1,-1) circle (.1cm);
\draw (.7,.7) node[anchor=south west] {$a$};
\draw (-1,-1) node[anchor=north east] {$W_3$};
\draw (1.55,0) node[anchor=west] {$W_1$};
\draw (0,1.55) node[anchor=south] {$W_2$};
\draw (0,-.1) node[anchor=south east] {$W_4$};
\end{tikzpicture}
      \end{minipage}
\]

The result and its proof are similar to Theorem \ref{thm: matrix half size}:
\begin{theorem}\label{thm: absorb big leg}
 Let $a \geq 2$. Then
 \begin{align*}
\uR(T_1)&=4(a+2),\\ 
\uR(T_2)&= a,
\end{align*}
and
\[
 \uR(T_1 \oplus T_2) = 4(a+2)+1.
\]
\end{theorem}
\begin{proof}
Let $V_1 = V_2 = \bbC^2$, $V_3 = \bbC^{a+2}$ and $V_4 = \bbC^{4(a+2)}$ so that $T_1 \in V_1 \ootimes V_4$. For $p=1,2$, let $\{ v^p_1, v^p_2\}$ be a basis of $V_p$ and let $\{ v^3_j : j =-1 \vvirg a\}$ be a basis of $V_3$. 

Similarly, let $W_1 = W_2 = \bbC^a$, $W_3 = W_4 = \bbC^1$, so that $T_2 \in W_1 \ootimes W_4$. For $p = 1,2$, let $\{ w^p_\ell : \ell = 1 \vvirg a\}$ be a basis of $W_p$ and let $\{ w^3 \}$ be a basis of $W_3$.

Regard $T_1 \oplus T_2$ as a tensor in $(V_1 \oplus W_1) \ootimes (V_4 \oplus W_4)$.  

The values of $\uR(T_1)$ and $\uR(T_2)$ are immediate. The lower bound $\uR(T_1 \oplus T_2) \geq 4(a+2)+1$ follows by conciseness.

For the upper bound, we determine a set of $4(a+2)+1$ rank-one elements $\calZ_\eps \subseteq (V_1 \oplus W_1) \otimes (V_2 \oplus W_2) \otimes (V_3 \oplus W_3)$ such that $T(V_4^* \oplus W_4^*) \subseteq \lim \langle \calZ_\eps \rangle$. By Proposition \ref{prop: rank and border rank via image of flattening}, this provides the desired upper bound.

Note 
\[
 T(V_4^* \oplus W_4^*) = V_1 \otimes V_2 \otimes V_3 + \langle \bfu(a) \rangle
\]
where, as in the proof of Theorem \ref{thm: matrix half size}, $\bfu(a) = \sum_1^a w^1_j \otimes w^2_j \otimes w^3 = \bfu_2(a) \otimes w^3$.

We will denote elements of $\calZ_\eps$ using indices $\{ -1, (1,1,-1) \vvirg (2,2,a)\}$. We drop the dependency from $\eps$ in the notation.

Define the following tensors:
\begin{align*}
& Z_{1,1,i} = ( v^1_1 + \eps  w^1_i  ) \otimes ( v^2_1 + \eps w^2_i ) \otimes  ( v^3_i + \eps w^3 ) &\quad \text{for $i=1 \vvirg a$}, \\
& Z_{1,2,i} = ( v^1_1 + \eps  w^1_i  ) \otimes ( v^2_2  - \eps w^2_i) \otimes   v^3_i  &\quad \text{for $i=1 \vvirg a$}, \\
& Z_{2,1,i} = ( v^1_2 - \eps  w^1_i  ) \otimes  v^2_1  \otimes v^3_i  &\quad \text{for $i=1 \vvirg a$}, \\
& Z_{2,2,i} = ( v^1_2 - \eps  w^1_i  ) \otimes  v^2_2  \otimes v^3_i  &\quad \text{for $i=1 \vvirg a$}, \\
 ~\\
& Z_{1,1,-1} =  v^1_1  \otimes  ( v^2_1 + \textfrac{2\eps}{a} \textsum_1^a w^2_j) \otimes (v^3_{-1} - \textfrac{a\eps}{2} w^3), \\
& Z_{1,1,0} = ( v^1_1 + \textfrac{2\eps}{a} \textsum_1^a w^1_j) \otimes   v^2_1  \otimes (v^3_{0} - \textfrac{a\eps}{2} w^3) ,\\
~\\
&Z_{1,2,-1} = v^1_1  \otimes  ( v^2_2 - \textfrac{2\eps}{a} \textsum_1^a w^2_j) \otimes v^3_{-1},\\
&Z_{2,1,0} = ( v^1_2 - \textfrac{2\eps}{a} \textsum_1^a w^1_j) \otimes   v^2_1  \otimes v^3_{0} ,\\
~\\
&Z_{2,1,-1} = v^1_2 \otimes v^2_1 \otimes v^3_{-1} \qquad \qquad Z_{1,2,0} = v^1_1 \otimes v^2_2 \otimes v^3_{0} , \\
&Z_{2,2,-1} = v^1_2 \otimes v^2_2 \otimes v^3_{-1} \qquad \qquad Z_{2,2,0} = v^1_2 \otimes v^2_2 \otimes v^3_{0}  .\\
\end{align*}
Finally, let $Z_{-1} = (v^1_1 + v^1_2) \otimes (v^2_1+v^2_2) \otimes (\textsum_{-1}^a v^3_i)$.

A direct calculation shows that $\sum_{(i_1,i_2,i_3) \in \{ (1,1,-1) \vvirg (2,2,a)\}} Z_{i_1,i_2,i_3} = Z_{-1} + \eps^3 \bfu(a)$, similarly to the proof of Theorem \ref{thm: matrix half size}.  

Let $\calZ_\eps =  \{ Z_{-1} , Z_{1,1,-1} \vvirg Z_{2,2,a}\} \subseteq (V_1 \oplus W_1) \otimes (V_2 \otimes W_2) \otimes (V_3 \otimes W_3)$: then $\calZ_\eps$ contains $4a+1$ elements. Let $E_\eps = \langle \calZ_\eps \rangle$ and $E_0 = \lim_{\eps \to 0} E_\eps$, where the limit is taken in the corresponding Grassmannian.

We show that $(T_1 \oplus T_2) (V_4^* \oplus W_4^*) \subseteq E_0$ (and in fact equality holds).

For every $(i_1,i_2,i_3) \in \{ (1,1,-1) \vvirg (2,2,a)\}$, we have $Z_{i_1,i_2,i_3}|_{\eps = 0} = v^1_{i_1} \otimes v^2_{i_2} \otimes v^3_{i_3}$. This shows $V_1 \otimes V_2 \otimes V_3 \subseteq E_0$.

Moreover $\eps^3 \bfu(a) = \sum_{i_1,i_2,i_3} Z_{i_1,i_2,i_3} - Z_{-1}$, therefore $\bfu(N) \in E_\eps$ for every $\eps$, hence $\bfu(N) \in E_0$.

This shows $(T_1 \oplus T_2) (V_4^* \oplus W_4^*) \subseteq E_0$ and concludes the proof.
\end{proof}

\subsection*{Construction 3: Adding a matrix, III}

This third construction deals with tensors of any order. Furthermore, for large dimensions, it provides an upper bound which improves on the trivial additive upper bound by a factor of $2$, as in Sch\"onhage's construction, unlike Constructions 1 and 2 which provide a saving of a factor of $3/2$ and $5/4$ respectively.

Fix $d \geq 2$ and $n_1 \vvirg n_d$. Let $N \leq n_1 \cdots n_d$. Consider the following two tensors:
\begin{equation}\label{eqn: tensors many legs}
   T_1 \quad = \qquad \begin{minipage}{.3\textwidth}
       \begin{tikzpicture}[scale=1]
\draw (0,1.5)-- (0,0);
\draw (0,0)-- (1.5,0);
\draw[fill=black] (1.5,0) circle (.1cm);
\draw[fill=black] (0,1.5) circle (.1cm);
\draw[fill=black] (0,0) circle (.1cm);
\draw (.75,0) node[anchor=north] {$n_1$};
\draw (0,.75) node[anchor=east] {$n_2$};
\draw (0,-.75) node[anchor=west] {$n_d$};
\draw (1.55,0) node[anchor=west] {$V_1$};
\draw (0,1.55) node[anchor=south] {$V_2$};
\draw (0,-.1) node[anchor=south east] {$V_{d+1}$};
\draw (-1.5,.05) node[anchor=south east] {$V_3$};
\draw (.05,-1.5) node[anchor=north west] {$V_d$};
\draw (0,0)-- (-1.5,0);
\draw (0,0)-- (0,-1.5);
\draw (0,0)-- (-1.47115  ,-0.2926);
\draw (0,0)-- (-1.3858 , -0.5740);
\draw (0,0)-- (-0.2926, -1.47115);
\draw (0,0)-- (-0.5740, -1.3858);
\draw[fill=black] (-1.5,0) circle (.08cm);
\draw[fill=black] (0,-1.5) circle (.08cm);
\draw[fill=black] (-1.47115  ,-0.2926) circle (.08cm);
\draw[fill=black] (-0.2926,-1.47115  ) circle (.08cm);
\draw[fill=black] (-1.3858 , -0.5740) circle (.08cm);
\draw[fill=black] ( -0.5740,-1.3858) circle (.08cm);
\draw (-1,-1) node[anchor= south west] {$\ddots$};
\end{tikzpicture}
      \end{minipage}
      \qquad
        T_2 \quad = \qquad \begin{minipage}{.3\textwidth}
       \begin{tikzpicture}[scale=1]
\draw (0,1.5)-- (1.5,0);
\draw[fill=black] (1.5,0) circle (.1cm);
\draw[fill=black] (0,1.5) circle (.1cm);
\draw[fill=black] (0,0) circle (.1cm);
\draw (.7,.7) node[anchor=south west] {$N$};
\draw (1.55,0) node[anchor=west] {$W_1$};
\draw (0,1.55) node[anchor=south] {$W_2$};
\draw (0,0) node[anchor=south] {$W_{d+1}$};
\draw (-1.5,.05) node[anchor=south east] {$W_3$};
\draw (.05,-1.5) node[anchor=north west] {$W_d$};
\draw[fill=black] (-1.5,0) circle (.08cm);
\draw[fill=black] (0,-1.5) circle (.08cm);
\draw[fill=black] (-1.47115  ,-0.2926) circle (.08cm);
\draw[fill=black] (-0.2926,-1.47115  ) circle (.08cm);
\draw[fill=black] (-1.3858 , -0.5740) circle (.08cm);
\draw[fill=black] ( -0.5740,-1.3858) circle (.08cm);
\draw (-1.3,-1.3) node[anchor= south west] {$\ddots$};
\end{tikzpicture}
      \end{minipage}
\end{equation}

For the sake of notation, we state and prove the following result in the special case $n := n_1 = \cdots = n_d$. A similar upper bound holds in general.
\begin{theorem}\label{thm: matrix on many legs}
Let $n , N,d \in \bbN$ be integers with $N \leq n^d$. Let $T_1 , T_2$  be as in \eqref{eqn: tensors many legs}. Then
 \begin{align*}
\uR(T_1)&= n^d,\\ 
\uR(T_2)&= N,
\end{align*}
and
\[
 \uR(T_1 \oplus T_2) \leq n^d + 2n^{d-1} + n^2(n+1)^{d-3} + 1 = n^d + \calO(n^{d-1}).
\]
\end{theorem}
\begin{proof}
We prove the result for $N = n^d$. The general result follows by semicontinuity of border rank.

For $p = 1 \vvirg d$, let $V_p = \bbC^{n}$ and $\{ v^p_{i_p} : i_p = 1 \vvirg n\}$ be a basis of $V_p$. Let $V_{d+1} = \bbC^{n^d}$, with basis $\{ v^{d+1}_{i_1 \vvirg i_d} : i_p = 1 \vvirg n\}$. Let $W_1 = W_2 = \bbC^N$ with basis $\{ w^1_{j_1 \vvirg j_d} : j_p = 1 \vvirg n\}$ of $W_1$ and similarly for $W_2$. For $p = 3 \vvirg d+1$, let $W_p = \bbC^1$ and let $w^p$ be a spanning vector of $W_p$. 

Regard $T_1 \oplus T_2$ as a tensor in $(V_1 \oplus W_1) \ootimes (V_{d+1} \oplus W_{d+1})$.

The values of $\uR(T_1)$ and $\uR(T_2)$ are immediate.

We present a border rank decomposition of $T_1 \oplus T_2$ providing the desired upper bound. 

For $i_1 \vvirg i_d$, define
\begin{align*}
 q_{i_1 \vvirg i_d}(\eps) = & ( v^1_{i_1} +\eps^{d-1} w^1_{i_1 \vvirg i_d}) \otimes  ( v^2_{i_2} +\eps^{d-1} w^2_{i_1 \vvirg i_d}) \otimes \\
 & (\eps v^3_{i_3} + w^3) \ootimes (\eps v^{d}_{i_d} + w^d)  \otimes (\eps^d v^{d+1}_{i_1 \vvirg i_d} + w^{d+1}).
\end{align*}

Define $Q(\eps) = \sum_{i_1 \vvirg i_d} q_{i_1 \vvirg i_d}(\eps)$ and note that $\rmR(Q(\eps)) \leq n^d$. Expand $Q(\eps)$ in $\eps$, writing $Q(\eps) = Q_0 + \eps Q_1 + \cdots + \eps^{2d-2} Q_{2d-2} + \mathit{h.o.t.}$ where $\mathit{h.o.t.}$ denotes the sum of higher order (in $\eps$) terms. 

\begin{claim}
We have $Q_{2d-2} = T_1 \oplus T_2$.  
\end{claim}
\begin{proof}[Proof of Claim \theclaim]
In each $q_{i_1 \vvirg i_d}(\eps)$, terms of degree $2d-2$ in $\eps$ arise in two possible ways: 
\begin{itemize}
 \dotitem the tensor product of all the $w$ terms, having degree $d-1$ on the first and second factor and degree $0$ on other factors;
 \dotitem the tensor product of all the $v$ terms, having degree $0$ on first and second factor degree $1$ on factors from $3$ to $d$ (total is degree $d-2$) and degree $d$ on factor $d+1$.
\end{itemize}
All other combinations have degree different from $2d-2$ and this proves the claim.
\end{proof}

We will provide the upper bound 
\[
\rmR( \textsum_0^{2d-3} \eps ^i Q_i)  \leq 2n^{d-1} + n^2(n+1)^{d-3} + 1 .
\]

\begin{claim}
Let $P(\eps) = \sum_0^{d-2} \eps^i Q_i$. Then $\rmR(P(\eps)) =1$.
\end{claim}
\begin{proof}[Proof of Claim \theclaim]
Observe $P(\eps) = \sum p_{i_1 \vvirg i_d}(\eps)$ where 
\[
  p_{i_1 \vvirg i_d}(\eps) =   v^1_{i_1} \otimes  v^2_{i_2} \otimes  (\eps v^3_{i_3} + w^3) \ootimes (\eps v^{d}_{i_d} + w^d)  \otimes  w^{d+1}.
\]
Therefore 
\[
 P(\eps) = (\textsum_{i_1} v^1_{i_1})\otimes (\textsum_{i_2}v^2_{i_2} ) \otimes (\textsum_{i_3}(\eps v^3_{i_3} + w^3 )) \ootimes ( \textsum_{i_{d-1}}(\eps v^{d}_{i_{d}} + w^{d})) \otimes w^{d+1} ,
\]
so that $\rmR(P(\eps)) = 1$.
\end{proof}

For $k=d-1\vvirg 2d-3$, write $Q_k = Q_k' + Q_k''$ where $Q'_k\in (V_1 \oplus W_1) \ootimes (V_d \oplus W_d)\otimes W_{d+1}$ and $Q_{k}''\in (V_1 \oplus W_1) \ootimes(V_d \oplus W_d)\otimes V_{d+1}$.
Note that $Q_{d-1}'' = 0$ because the component of the last factor of $q_{i_1 \vvirg i_d}$ on $V_{d+1}$ is $\eps^d v^{d+1}_{i_1 \vvirg i_{d}}$.

 \begin{claim}
 Let $P'(\eps) = \sum_{d-1}^{2d-3} \eps^{k} Q'_k$. Then $\rmR(P'(\eps)) \leq 2 n^{d-1}$.
\end{claim}
\begin{proof}[Proof of Claim \theclaim]
Observe
 \begin{align*}
  P'(\eps) &= \sum_{i_1,i_3 \vvirg i_{d}}  v^1_{i_1} \otimes \left(\eps^{d-1}   \textsum_{i_2} w^2_{i_1 \vvirg i_{d}}\right) \otimes (\eps v^3_{i_3} +w^3) \ootimes (\eps v^{d}_{i_{d}} + w^{d}) \otimes w^{d+1}  \\
  &+\sum_{i_2,i_3 \vvirg i_{d}}  \left( \eps^{d-1} \textsum_{i_1} w^1_{i_1 \vvirg i_{d}} \right) \otimes v^2_{i_2}  \otimes (\eps v^3_{i_3} +w^3) \ootimes (\eps v^{d}_{i_{d}} + w^{d}) \otimes w^{d+1}.
 \end{align*}
This gives the upper bounds $n^{d-1}$ for each one of the two summations above. Adding the two contributions together, we obtain the desired upper bound.
\end{proof}
\begin{claim}
 For every $k = 0 \vvirg d -3 $, $\rmR(Q''_{d+k}) \leq \binom{d-3}{k} n^{k+2}$.
\end{claim}
\begin{proof}[Proof of Claim \theclaim]
Every term of $Q''_{d+k}$ arises in $q_{i_1 \vvirg i_d}$ as the projection on $V_1 \otimes V_2 \otimes U_3 \ootimes U_d \otimes V_{d+1}$ where exactly $k$ among $U_3 \vvirg U_d$ are equal to the corresponding $V_j$ and the other $d-3-k$ are equal to the corresponding $W_j$. In particular, we have 
\[
 Q''_{d+k} = \sum_{\substack{|J| \subseteq \{ 3 \vvirg d\} \\ |J| = k}} \sum_{\substack{i_1, i_2 \\ (i_j = 1 \vvirg n : j \in J)}} v^1_{i_1} \otimes v^2_{i_2} \otimes \bigotimes_{j \in J} v^{j}_{i_j} \otimes \bigotimes_{j' \notin J} w^{j'} \otimes \left(\textsum_{(i_{j'} : j' \notin J)} v_{i_1 \vvirg i_k}\right).
\]
From this expression, we deduce $\rmR(Q''_{d+k}) \leq \binom{d-3}{k} n^{k+2}$.
\end{proof}

Setting $P''(\eps) = \sum_{k = d} ^{2d-3} \eps^k Q''_k$, Claim 4 provides $\rmR(P''(\eps)) \leq \sum_{\kappa = 0}^{d-3} \binom{d-3}{\kappa} n^{\kappa+2} = n^2 (n+1)^{d-3}$.

We conclude that
\begin{align*}
\rmR( \textsum_0^{2d-3} \eps ^i Q_i) = &\rmR(P(\eps)+P'(\eps)+P''(\eps))  \\ \leq & \rmR(P(\eps)) + \rmR(P'(\eps)) + \rmR(P''(\eps)) \leq 1 + 2 n^{d-1} + n^2 (n+1)^{d-3}.
\end{align*}
This concludes the proof, because $T_1 \oplus T_2 = \lim_{\eps \to 0}\frac{1}{\eps^{2d-2}}\bigl[ Q(\eps) - (P(\eps) + P'(\eps)+P''(\eps))\bigr]$, giving the upper bound on the border rank
\begin{align*}
 \uR(T_1 \oplus T_2) \leq &\uR(Q(\eps)) + \uR(P(\eps)+P'(\eps)+P''(\eps)) \leq  n^d + 1 + 2 n^{d-1} + n^2 (n+1)^{d-3}.
\end{align*}
\end{proof}

\subsection*{Construction 4: Adding a higher order tensor}

The last construction deals with tensors of order $4$. Fix integers $n_1,n_2,n_3$. For integers $a,b$ let $[a,b] = \{ a,a+1\vvirg b\}$. Consider the two tensors

\[
    T_1 \quad = \qquad \begin{minipage}{.3\textwidth}
       \begin{tikzpicture}[scale=1]
\draw (0,1.5)-- (0,0);
\draw (0,0)-- (1.5,0);
\draw (0,0)-- (-1,-1);
\draw[fill=black] (1.5,0) circle (.1cm);
\draw[fill=black] (0,1.5) circle (.1cm);
\draw[fill=black] (0,0) circle (.1cm);
\draw[fill=black] (-1,-1) circle (.1cm);
\draw (.75,0) node[anchor=north] {$n_1 +1$};
\draw (0,.75) node[anchor=east] {$n_2 +1$};
\draw (-.5,-.5) node[anchor=south east] {$n_3 +1$};
\draw (-1,-1) node[anchor=north east] {$V_3$};
\draw (1.55,0) node[anchor=west] {$V_1$};
\draw (0,1.55) node[anchor=south] {$V_2$};
\draw (0,-.1) node[anchor=south east] {$V_4$};
\end{tikzpicture}
      \end{minipage}
      \qquad
        T_2 \quad = \qquad \begin{minipage}{.3\textwidth}
       \begin{tikzpicture}[scale=1]
 \draw [rounded corners=2mm,fill=gray!50] (1.8,0)--(0,1.8)--(-1.3,-1.3)--cycle;
 \draw [rounded corners=2mm,fill=white] (1.3,0)--(0,1.3)--(-0.8,-0.8)--cycle;
\draw[fill=black] (1.5,0) circle (.1cm);
\draw[fill=black] (0,1.5) circle (.1cm);
\draw[fill=black] (0,0) circle (.1cm);
\draw[fill=black] (-1,-1) circle (.1cm);
\draw (-1.1,-1) node[anchor=north east] {$W_3$};
\draw (1.55,0) node[anchor=west] {$W_1$};
\draw (0,1.55) node[anchor=south] {$W_2$};
\draw (-.05,0) node[anchor=south] {$W_4$};
\draw (0.7,-0.8) node[anchor=west] {$M$};
\end{tikzpicture}
\end{minipage}
\]
where $M = M(n_1,n_2,n_3)$ is the maximum possible integer such that the following combinatorial independence condition holds. There exist four disjoint subsets $J, K_1,K_2,K_3$ of $[n_1] \times [n_2] \times [n_3]$, all of order $M$ such that there are three bijections $s_i : J \to K_i$ fixing the $i$-th component, in the sense that if $s_i (j_1,j_2,j_3) = (k_1 ,k_2,k_3)$  then $j_i = k_i$.

\begin{lemma}\label{lemma:best for four factors}
 Let $n_1,n_2,n_3$ be even. Then $M(n_1,n_2,n_3) = \frac{1}{4}n_1n_2n_3$.
\end{lemma}
\begin{proof}
 Let $m_i = \frac{1}{2} n_i$. Define
 \begin{align*}
J' &= [m_1] \times [m_2] \times [m_3] ,\\ 
K'_1 &= [m_1] \times [m_2+1,n_2] \times [m_3+1,n_3] \quad s_1 ( j_1,j_2,j_3) = (j_1 ,m_2+ j_2, m_3+ j_3),\\
K'_2 &= [m_1+1,n_1] \times [m_2] \times [m_3+1,n_3] \quad s_2 ( j_1,j_2,j_3) = (m_1+ j_1 ,j_2, m_3+ j_3),\\
K'_3 &= [m_1+1,n_1] \times [m_2+1,n_2] \times [m_3] \quad s_3 ( j_1,j_2,j_3) = (m_1+j_1 ,m_2+j_2, j_3), \\
J'' &= [m_1+1,n_1] \times [m_2+1,n_2] \times [m_3+1,n_3], \\ 
K''_1 &= [m_1+1,n_1] \times [m_2] \times [m_3] \quad s_1 ( j_1,j_2,j_3) = (j_1 ,-m_2+ j_2, -m_3+ j_3),\\
K''_2 &= [m_1] \times [m_2+1,n_2] \times [m_3] \quad s_2 ( j_1,j_2,j_3) = (-m_1+ j_1 ,j_2, -m_3+ j_3),\\
K''_3 &= [m_1] \times [m_2] \times [m_3+1,n_3] \quad s_3 ( j_1,j_2,j_3) = (-m_1+j_1 ,-m_2+j_2, j_3).
\end{align*}
The position of $J', K'_1,K'_2,K'_3$ is represented in Figure \ref{figure: cubes} as three subsets of the set $[n_1]\times [n_2]\times [n_3]$. Let $J = J' \sqcup J''$ and $K_i = K'_i \sqcup K_i''$. It is immediate to verify that this satisfies the required conditions. Moreover $\# J = \#J' + \#J'' =2 \cdot \frac{n_1}{2} \cdot \frac{n_2}{2} \cdot \frac{n_3}{2} = 2 \cdot \frac{n_1n_2n_3}{8} = \frac{n_1n_2n_3}{4}$.

 \begin{figure}[!htp]
 \includegraphics[scale=.7]{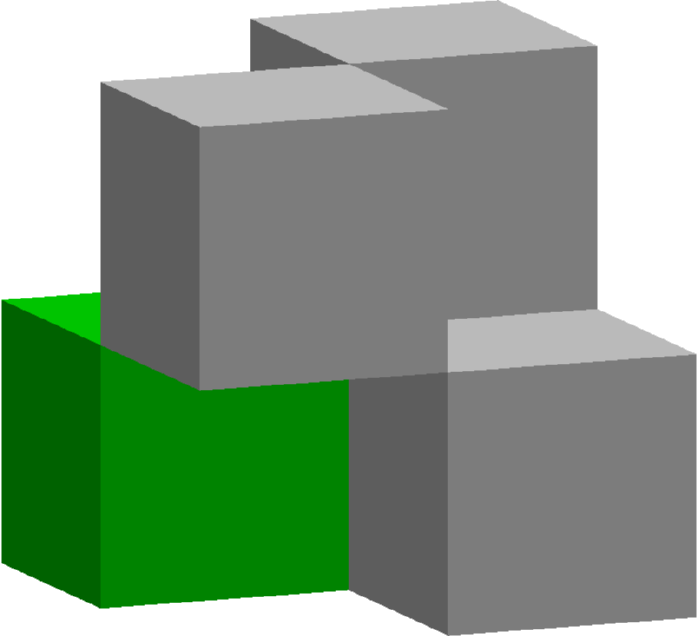} 
 \caption{Schematic representation of $J'$ (green) and $K_1',K_2',K_3'$ (gray) in the proof of Lemma \ref{lemma:best for four factors}. $J'',K_1'',K_2'',K_3''$ are represented by the three complementary cubes.}\label{figure: cubes}
 \end{figure}

\end{proof}

The proof of the following result is similar to the one of Theorem \ref{thm: matrix half size}.
\begin{theorem}\label{thm: absorbing GHZ}
Fix $n_1,n_2,n_3$ and let $M = M(n_1,n_2,n_3)$. Then 
\begin{align*}
\uR(T_1)&=(n_1+1)(n_2+1)(n_3+1),\\ 
\uR(T_2)&= M,
\end{align*}
and
\[
 \uR(T_1 \oplus T_2) = (n_1+1)(n_2+1)(n_3+1)+ 1.
\]
\end{theorem}
\begin{proof}
For $ p =1 ,2,3$, let $V_p = \bbC^{n_p + 1}$ and $V_4 = \bbC^{(n_1+1)(n_2+1)(n_3+1)}$ so that $T_1 \in V_1 \ootimes V_4$. Let $\{ v^p_j : j = 0 \vvirg n_p\}$ be a basis of $V_p$ and let $\{ v^4_{j_1,j_2,j_3} : j_p = 0 \vvirg  n_p\}$ be a basis of $V_4$. We have $T_1 \in V_1 \ootimes V_4$.

Similarly, for $p= 1,2,3$, let $W_p = \bbC^{M}$ and $W_4 = \bbC^1$. Let $\{ w^p_\ell : \ell = 1 \vvirg M\}$ be a basis of $W_p$ and let $w^4$ be a spanning vector of $W_4$. We have $T_2 \in W_1 \ootimes W_4$. 

Regard $T_1 \oplus T_2$ as a tensor in $(V_1 \oplus W_1) \ootimes (V_4 \oplus W_4)$.

The values of $\uR(T_1)$ and $\uR(T_2)$ are immediate. The lower bound $\uR(T_1 \oplus T_2) \geq (n_1+1)(n_2+1)(n_3+1)+ 1$ follows by conciseness.
 
For the upper bound, we determine a set of $(n_1+1)(n_2+1)(n_3+1)+ 1$ rank-one elements $\calZ_\eps \subseteq (V_1 \oplus W_1) \otimes (V_2 \oplus W_2) \otimes (V_3 \oplus W_3)$ such that $T(V_4^* \oplus W_4^*) = \lim \langle \calZ_\eps \rangle$. By Proposition \ref{prop: rank and border rank via image of flattening}, this provides the desired upper bound. 

Note
\[
T(V_4^* \oplus W_4^*) = V_1 \otimes V_2 \otimes V_3 \oplus \langle \bfu_3(M) \rangle,
\]
where $\bfu_3(M) = \textsum_1^M w^{1}_\ell \otimes w^{2}_\ell \otimes w^{3}_\ell \in W_1 \otimes W_2 \otimes W_3$.

We denote the elements of $\calZ_\eps$ using indices $\{ -1, (0,0,0) \vvirg (n_1,n_2,n_3)\}$. We drop the dependency from $\eps$ in the notation.

Let $J,K_1,K_2,K_3$ be the subsets determining $M = M(n_1,n_2,n_3)$ and let $s_p : J \to K_p$ be the three fixed bijections.

Define bijections $\bfj : J \to [1 , M ]$ and $\bfk_p : K_p \to [1,M]$ for $p=1,2,3$ commuting with the fixed $s_i$'s namely $\bfj = \bfk_p \circ s_p$.

If $(j_1,j_2,j_3) \in J$, define 
 \[
Z_{j_1,j_2,j_3} = (v^1_{j_1} + \eps w^1_{\bfj(j_1,j_2,j_3)}) \otimes   (v^2_{j_2} + \eps w^2_{\bfj(j_1,j_2,j_3)}) \otimes (v^3_{j_3} + \eps w^3_{\bfj(j_1,j_2,j_3)}) 
 \]

The component of degree $3$ (with respect to $\eps$) in $\sum_{(j_1,j_2,j_3) \in J} Z_{j_1,j_2,j_3}$ is exactly $\bfu_3(M)$.

If $(k_1,k_2,k_3) \in K_1$, define
\[
 Z_{k_1,k_2,k_3} = v^1_{k_1} \otimes  (v^2_{k_2} + \eps w^2_{\bfk_1(k_1,k_2,k_3)}) \otimes (v^3_{k_3} - \eps w^3_{\bfk_1(k_1,k_2,k_3)} ).
\]

If $(k_1,k_2,k_3) \in K_2$, define
\[
 Z_{k_1,k_2,k_3} = ( v^1_{k_1} - \eps w^1_{\bfk_2(k_1,k_2,k_3)} ) \otimes v^2_{k_2}  \otimes (v^3_{k_3} + \eps w^3_{\bfk_2(k_1,k_2,k_3)} ).
\]

If $(k_1,k_2,k_3) \in K_3$, define
\[
 Z_{k_1,k_2,k_3} = ( v^1_{k_1} + \eps w^1_{\bfk_3(k_1,k_2,k_3)} ) \otimes (v^2_{k_2} - \eps w^2_{\bfk_3(k_1,k_2,k_3)})  \otimes v^3_{k_3} .
\]

The component of degree $2$ of $\sum_{(k_1,k_2,k_3) \in K_1 \sqcup K_2 \sqcup K_3} Z_{k_1,k_2,k_3}$ is opposite to the component of degree $2$ of $\sum_{(j_1,j_2,j_3) \in J} Z_{j_1,j_2,j_3}$. Indeed, the term of the form $v^1_{j_1}\otimes\eps w^2_{\bfj(j_1,j_2,j_3)}\otimes \eps w^3_{\bfj(j_1,j_2,j_3)}$ is opposite to $v^1_{k_1}\otimes\eps w^2_{\bfk_1(k_1,k_2,k_3)}\otimes(- \eps w^3_{\bfk_1(k_1,k_2,k_3)})$ for $(k_1,k_2,k_3)=s_1(j_1,j_2,j_3)$ etc. As a consequence, setting 
\[
S = \sum_{(k_1,k_2,k_3) \in K_1 \sqcup K_2 \sqcup K_3} Z_{k_1,k_2,k_3} + \sum_{(j_1,j_2,j_3) \in J} Z_{j_1,j_2,j_3}, 
\]
we deduce that the component of degree $2$ of $S$ is $0$. 

Write $S = S_0 + \eps S_1 + \eps^3 \bfu_3(M)$ and 
\begin{align*}
S_1 = & \sum_{i_1,i_2} v^1_{i_1} \otimes v^2_{i_2} \otimes \omega^3_{i_1,i_2}   + \sum_{i_1,i_3} v^1_{i_1} \otimes \omega^2_{i_1,i_3} \otimes v^3_{i_3}  + \sum_{i_2,i_3} \omega^1_{i_2,i_3} \otimes v^2_{i_2} \otimes v^3_{i_3} 
\end{align*}
where $\omega^1_{i_2,i_3} \in W_1, \omega^2_{i_1,i_3} \in W_2, \omega^3_{i_1,i_2} \in W_3$.

For $i_1 = 1\vvirg n_1$, $i_2 \in 1 \vvirg n_2$, $i_3 =1\vvirg n_3$, define
\begin{align*}
 Z_{0,i_2,i_3} &= (v^1_0 - \eps \omega^1_{i_2,i_3}) \otimes v^2_{i_2} \otimes v^3_{i_3}, \\
 Z_{i_1,0,i_3} &= v^1_{i_1} \otimes (v^2_0 - \eps \omega^2_{i_1,i_3}) \otimes v^3_{i_3}, \\
 Z_{i_1,i_2,0} &= v^1_{i_1} \otimes v^2_{i_2} \otimes (v^3_0 - \eps \omega^3_{i_1,i_2}). \\
\end{align*}
By construction, $S + \textsum_{i_2,i_3} Z_{0,i_2,i_3} + \textsum_{i_1,i_3} Z_{i_1,0,i_3} + \textsum_{i_1,i_2} Z_{i_1,i_2,0}$ is $0$ in degrees $1$ and $2$ and $\bfu_3(M)$ in degree $3$.

Define 
\[
\Omega = [0,n_1] \times [0,n_2] \times [0,n_3] \setminus ( J \sqcup K_1 \sqcup K_2 \sqcup K_3 \sqcup L) 
\]
 where $L$ is the set of triples with exactly one zero. The triples in $\Omega$ are the ones for which a rank-one tensor $Z_{i_1,i_2,i_3}$ has yet to be defined.

For every $(i_1,i_2,i_3) \in \Omega$, define $Z_{i_1,i_2,i_3} = v^1_{i_1} \otimes v^2_{i_2} \otimes v^3_{i_3}$. It is immediate to verify 
\[
 \sum_{\substack{ i_p =0 \vvirg n_p \\ p = 1,2,3}} Z_{i_1,i_2,i_3} = Z_{-1} + \eps^3 \bfu_3(M).
\]
where $Z_{-1} = (\textsum_{i_1 =0}^{n_1} v^1_{i_1}) \otimes (\textsum_{i_2 =0}^{n_2} v^2_{i_2}) \otimes (\textsum_{i_3 =0}^{n_3} v^3_{i_3})$.

Therefore
\[
 \sum_{\substack{ i_p =0 \vvirg n_p \\ p = 1,2,3}} Z_{i_1,i_2,i_3} - Z_{-1} = \eps^3 \bfu_3(M).
\]

Let $\calZ_\eps = \{ Z_{-1}, Z_{0,0,0} \vvirg Z_{n_1,n_2,n_3}\} \subseteq (V_1 \oplus W_1) \otimes (V_2 \oplus W_2) \otimes (V_3 \oplus W_3)$: then $\calZ_\eps$ has $(n_1+1)(n_2+1)(n_3+1)+1$ elements. Let $E_\eps = \langle \calZ_\eps \rangle$ and let $E_0 = \lim_{\eps \to 0} E_\eps$, where the limit is taken in the corresponding Grassmannian.

We show that $(T_1 \oplus T_2)(V_4^* \oplus W_4^*) \subseteq E_0$ (and in fact equality holds).

For every $(i_1,i_2,i_3)$ we have $Z_{i_1,i_2,i_3}|_{\eps = 0} = v^1_{i_1} \otimes v^2_{i_2} \otimes v^3_{i_3}$. This shows $V_1 \otimes V_2 \otimes V_3 \subseteq E_0$. 

Moreover, $\eps^3 \bfu_3(M) = \sum_{i_1,i_2,i_3} Z_{i_1,i_2,i_3} - Z_{-1} \in E_\eps$, therefore $\bfu_3(M) \in E_\eps$ for every $\eps$. Hence $ \bfu_3(M) \in E_0$.

This shows $(T_1 \oplus T_2)(V_4^* \oplus W_4^*) \subseteq E_0$ and concludes the proof.
\end{proof}

\section{Consequences on the exponent of certain graph tensors}\label{section: exponent}

In this section, we use Construction 2, Construction 3 and Construction 4 to obtain upper bounds on the exponent of the graph tensors obtained as Kronecker products of the tensors $T_1$ and $T_2$ (or possibly their Kronecker powers) involved in the construction. Following Sch\"onhage's technique, we use the border rank upper bound on the direct sum (Theorem \ref{thm: absorb big leg}, Theorem \ref{thm: matrix on many legs} and Theorem \ref{thm: absorbing GHZ}) and Proposition \ref{prop: bound exponent general} to determine an upper bound on the asymptotic rank, and in turn on the exponent, of certain tensors.

We benchmark our results comparing them with the trivial upper bound of \eqref{eqn: trivial asy rank}. 

\subsection{Extended matrix multiplication}

We use the result of Theorem \ref{thm: absorb big leg} to obtain an upper bound on the exponent of the tensor
 \[
\EMamu(n_1,n_2,n_3; n_4 ) = \qquad \begin{minipage}{.3\textwidth}
       \begin{tikzpicture}[scale=1]
\draw (0,1.5)-- (0,0);
\draw (0,0)-- (1.5,0);
\draw (0,0)-- (-1,-1);
\draw (0,1.5)-- (1.5,0);
\draw (.75,0) node[anchor=north] { $n_1$};
\draw (0,.75) node[anchor=east] { $n_2$};
\draw (-.5,-.5) node[anchor=south east] { $n_4$};
\draw (.7,.7) node[anchor=south west] { $n_3$};
\draw[fill=black] (1.5,0) circle (.1cm);
\draw[fill=black] (0,1.5) circle (.1cm);
\draw[fill=black] (0,0) circle (.1cm);
\draw[fill=black] (-1,-1) circle (.1cm);
\end{tikzpicture}
      \end{minipage}
\]
for some instances of $n_1 \vvirg n_4$. We call this tensor \emph{extended matrix multiplication tensor} because it can be realized as Kronecker product of the matrix multiplication tensor and a \emph{dangling} matrix; graphically: 
\[
\EMamu(n_1,n_2,n_3; n_4 ) = \qquad \begin{minipage}{.3\textwidth}
       \begin{tikzpicture}[scale=1]
\draw (0,1.5)-- (0,0);
\draw (0,0)-- (1.5,0);
\draw (0,1.5)-- (1.5,0);
\draw (.75,0) node[anchor=north] { $n_1$};
\draw (0,.75) node[anchor=east] { $n_2$};
\draw (.7,.7) node[anchor=south west] { $n_3$};
\draw[fill=black] (1.5,0) circle (.1cm);
\draw[fill=black] (0,1.5) circle (.1cm);
\draw[fill=black] (0,0) circle (.1cm);
\draw[fill=black] (-1,-1) circle (.1cm);
\end{tikzpicture}
      \end{minipage} \hspace{-1cm}
      \boxtimes \quad
      \begin{minipage}{.3\textwidth}
       \begin{tikzpicture}[scale=1]
\draw (0,0)-- (-1,-1);
\draw (-.5,-.5) node[anchor=south east] { $n_4$};
\draw[fill=black] (1.5,0) circle (.1cm);
\draw[fill=black] (0,1.5) circle (.1cm);
\draw[fill=black] (0,0) circle (.1cm);
\draw[fill=black] (-1,-1) circle (.1cm);
\end{tikzpicture}
      \end{minipage} 
\]
The upper bound from \eqref{eqn: trivial asy rank} provides 
\[
 \omega(\EMamu(n_1,n_2,n_3;n_4)) \leq \log(n_1n_2n_3n_4) = \textsum \log(n_i).
\]

The extended matrix multiplication tensor can be realized as a Kronecker product of the tensor $T_1$ and $T_2$ of Construction 1 and Construction 2; indeed
\[
\EMamu(n_1,n_2,n_3; n_4 ) = \qquad \begin{minipage}{.3\textwidth}
       \begin{tikzpicture}[scale=1]
\draw (0,1.5)-- (0,0);
\draw (0,0)-- (1.5,0);
\draw (0,0)-- (-1,-1);
\draw (.75,0) node[anchor=north] { $n_1$};
\draw (0,.75) node[anchor=east] { $n_2$};
\draw (-.5,-.5) node[anchor=south east] { $n_4$};
\draw[fill=black] (1.5,0) circle (.1cm);
\draw[fill=black] (0,1.5) circle (.1cm);
\draw[fill=black] (0,0) circle (.1cm);
\draw[fill=black] (-1,-1) circle (.1cm);
\end{tikzpicture}
      \end{minipage} \hspace{-1cm}
      \boxtimes \quad
      \begin{minipage}{.3\textwidth}
       \begin{tikzpicture}[scale=1]
\draw (0,1.5)-- (1.5,0);
\draw (.7,.7) node[anchor=south west] { $n_3$};
\draw[fill=black] (1.5,0) circle (.1cm);
\draw[fill=black] (0,1.5) circle (.1cm);
\draw[fill=black] (0,0) circle (.1cm);
\draw[fill=black] (-1,-1) circle (.1cm);
\end{tikzpicture}
      \end{minipage} 
\]
Fix $n_1 = n_2=2$, and write $T_1(n_4)$, $T_2(n_3)$ for the two tensors above: in particular 
\[
\EMamu(2,2,n_3; n_4) = T_1(n_4) \boxtimes T_2(n_3) .
\]
We are going to use the result of Theorem \ref{thm: absorb big leg} to obtain an upper bound on the exponent $\omega(\EMamu(2,2,n_3; n_4))$.

\begin{theorem}\label{thm: omega ext mamu}
Let $a \geq 0$ and let $p \in (0,1)$. Then
\[
  \omega \left( \EMamu(2 , 2 , a^{\frac{1-p}{p}}; a+2) \right) \leq \frac{1}{p} \left[ \log (4(a+2)+1) - h(p) \right]. 
\]
\end{theorem}
\begin{proof}
By Theorem \ref{thm: absorb big leg}, for every $a \geq 2$, we have $\uR(T_1(a+2) \oplus T_2(a)) = 4(a+2)+1$.

For every $p \in (0,1)$, we have 
\[
T_1(a+2)^{\boxtimes p} \boxtimes T_2(a+2)^{\boxtimes(1- p)} =  \EMamu( 2^{p} , 2^{p} , a^{(1-p)} ; (a+2)^{p}) .
\]
 Therefore Proposition \ref{prop: bound exponent general} provides the upper bound
\[
 \omega( \EMamu( 2^{p} , 2^{p} , a^{(1-p)} ; (a+2)^{p}) ) \leq \log (4(a+2)+1) - h(p).
\]
Considering the Kronecker power with exponent $1/p$ on the left-hand side, we obtain the desired upper bound.
\end{proof}
Now, for every $n_3,n_4 \geq 2$, define 
\[
 a := a(n_3,n_4) = n_4 - 2  \qquad p := p(n_3,n_4) = \frac{\log(n_4-2)}{\log(n_3) + \log(n_4-2)}
\]
so that $n_3 = a^{\frac{1-p}{p}}$ and $n_4 = a+2$. Let $\omega_{\text{Sch}}(n_3,n_4) =  \frac{1}{p} \left[ \log (4(a+2)+1) - h(p) \right]$ be the upper bound of Theorem \ref{thm: omega ext mamu}, and let $\omega_{\text{triv}} = 2 + \log(n_3) + \log(n_4)$ be the trivial upper bound from \eqref{eqn: trivial asy rank}. We compare the two bounds in Figure \ref{figure: bound ext mamu} for $n_3 = 2 \vvirg 100$ and $n_4 = 4 \vvirg 100$. In particular, we observe that for $n_4 \gg n_3$, the upper bound from Theorem \ref{thm: omega ext mamu} obtained via the non-additivity construction is stronger than the trivial one.

We point out that one can obtain an upper bound on the exponent of the extended matrix multiplication tensor from upper bounds on the exponent of matrix multiplication. Indeed
\[
 \omega(\EMamu(n_1,n_2,n_3;n_4)) \leq \log(n_4) + \omega(\Mamu(n_1,n_2,n_3)).
\]
Applying the best known upper bounds on $\omega(\Mamu(n_1,n_2,n_3))$, one obtains stronger bounds on $ \omega(\EMamu(n_1,n_2,n_3;n_4))$ than the one of Theorem \ref{thm: omega ext mamu}. However, the method followed in this section is much simpler than the methods used to obtain upper bounds on $\omega(\Mamu(n_1,n_2,n_3))$ and yet it delivers nontrivial bounds in a wide range, as one can observe in Figure \ref{figure: bound ext mamu}.
 
\begin{figure}[!htp]
\includegraphics[width =.6\textwidth]{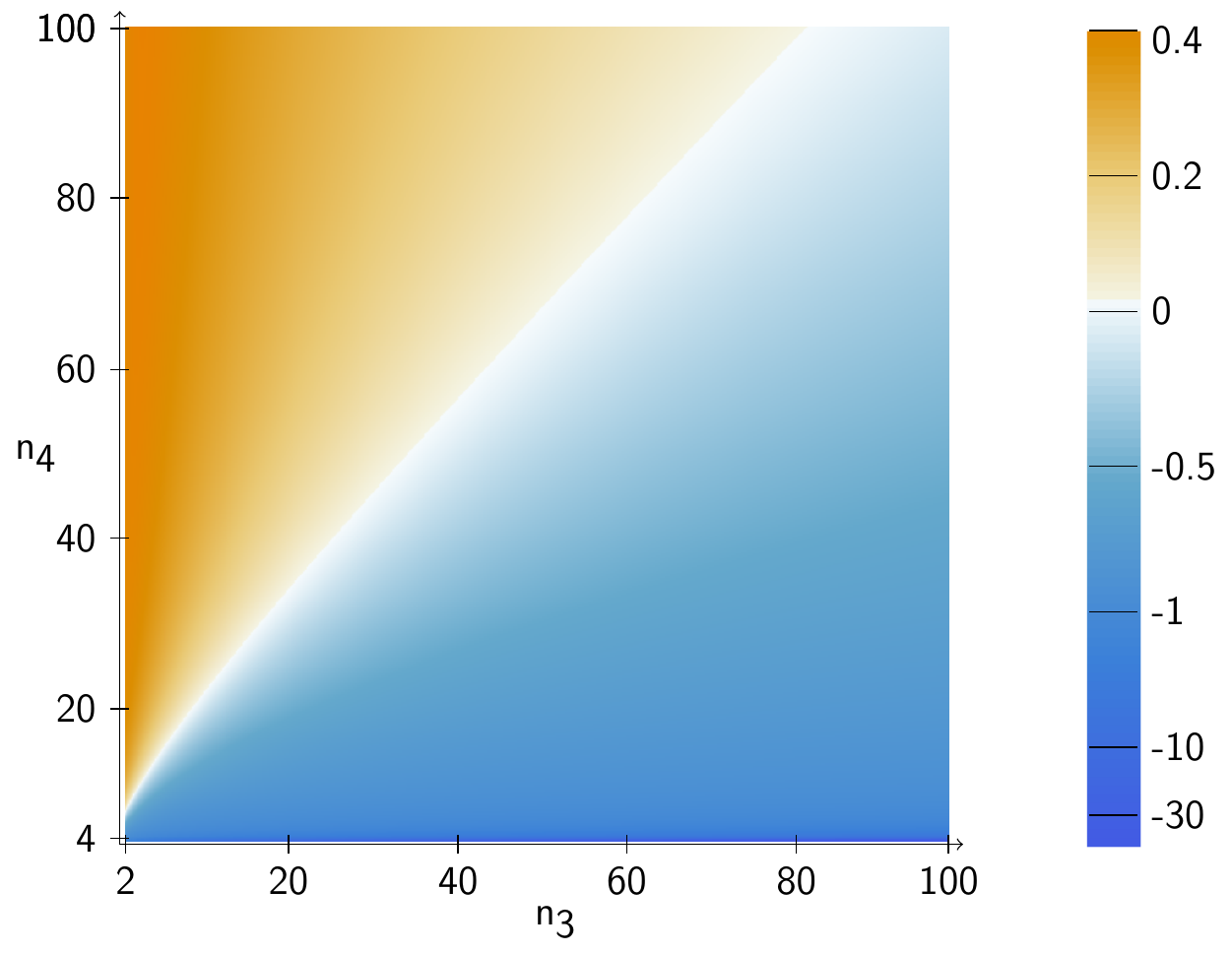}
\caption{Density graph of $\omega_{\triv} - \omega_{\text{Sch}}$ as a function of $n_3$ and $n_4$. The blue region corresponds to negative values (i.e., $\omega_{\triv} < \omega_{\text{Sch}}$); the orange region corresponds to positive values (i.e., $\omega_{\triv} > \omega_{\text{Sch}}$). Darker shades correspond to more extreme values.} \label{figure: bound ext mamu}
\end{figure}

\subsection{Multi-extended matrix multiplication}

We use the result of Theorem \ref{thm: matrix on many legs} to obtain an upper bound on the exponent of the tensor
\[
\multiEMamu(d;n,N) = \qquad \begin{minipage}{.3\textwidth}
        \begin{tikzpicture}[scale=1]
\draw (0,1.5)-- (0,0);
\draw (0,0)-- (1.5,0);
\draw[fill=black] (1.5,0) circle (.1cm);
\draw[fill=black] (0,1.5) circle (.1cm);
\draw[fill=black] (0,0) circle (.1cm);
\draw (.75,0) node[anchor=north] {$n$};
\draw (0,.75) node[anchor=east] {$n$};
\draw (0,-.75) node[anchor=west] {$n$};
\draw (1.2,1) node[anchor=east]{$N$};
\draw (0,0)-- (-1.5,0);
\draw (0,1.5)-- (1.5,0);
\draw (0,0)-- (0,-1.5);
\draw (0,0)-- (-1.47115  ,-0.2926);
\draw (0,0)-- (-1.3858 , -0.5740);
\draw (0,0)-- (-0.2926, -1.47115);
\draw (0,0)-- (-0.5740, -1.3858);
\draw[fill=black] (-1.5,0) circle (.08cm);
\draw[fill=black] (0,-1.5) circle (.08cm);
\draw[fill=black] (-1.47115  ,-0.2926) circle (.08cm);
\draw[fill=black] (-0.2926,-1.47115  ) circle (.08cm);
\draw[fill=black] (-1.3858 , -0.5740) circle (.08cm);
\draw[fill=black] ( -0.5740,-1.3858) circle (.08cm);
\draw (-1,-1) node[anchor= south west] {$\ddots$};
\end{tikzpicture}
\end{minipage}
\]
where the central vertex has degree $d$.

The tensor $\multiEMamu(d;n,N)$ can be realized as Kronecker product of the tensors $T_1$ and $T_2$ of Construction 3; indeed 
\begin{equation}
   \multiEMamu(d;n,N) \quad = \qquad \begin{minipage}{.3\textwidth}
       \begin{tikzpicture}[scale=1]
\draw (0,1.5)-- (0,0);
\draw (0,0)-- (1.5,0);
\draw[fill=black] (1.5,0) circle (.1cm);
\draw[fill=black] (0,1.5) circle (.1cm);
\draw[fill=black] (0,0) circle (.1cm);
\draw (.75,0) node[anchor=north] {$n$};
\draw (0,.75) node[anchor=east] {$n$};
\draw (0,-.75) node[anchor=west] {$n$};
\draw (0,0)-- (-1.5,0);
\draw (0,0)-- (0,-1.5);
\draw (0,0)-- (-1.47115  ,-0.2926);
\draw (0,0)-- (-1.3858 , -0.5740);
\draw (0,0)-- (-0.2926, -1.47115);
\draw (0,0)-- (-0.5740, -1.3858);
\draw[fill=black] (-1.5,0) circle (.08cm);
\draw[fill=black] (0,-1.5) circle (.08cm);
\draw[fill=black] (-1.47115  ,-0.2926) circle (.08cm);
\draw[fill=black] (-0.2926,-1.47115  ) circle (.08cm);
\draw[fill=black] (-1.3858 , -0.5740) circle (.08cm);
\draw[fill=black] ( -0.5740,-1.3858) circle (.08cm);
\draw (-1,-1) node[anchor= south west] {$\ddots$};
\end{tikzpicture}
      \end{minipage}
      \hspace{-1cm} 
      \boxtimes 
      \qquad
 \begin{minipage}{.3\textwidth}
       \begin{tikzpicture}[scale=1]
\draw (0,1.5)-- (1.5,0);
\draw[fill=black] (1.5,0) circle (.1cm);
\draw[fill=black] (0,1.5) circle (.1cm);
\draw[fill=black] (0,0) circle (.1cm);
\draw (.7,.7) node[anchor=south west] {$N$};
\draw[fill=black] (-1.5,0) circle (.08cm);
\draw[fill=black] (0,-1.5) circle (.08cm);
\draw[fill=black] (-1.47115  ,-0.2926) circle (.08cm);
\draw[fill=black] (-0.2926,-1.47115  ) circle (.08cm);
\draw[fill=black] (-1.3858 , -0.5740) circle (.08cm);
\draw[fill=black] ( -0.5740,-1.3858) circle (.08cm);
\draw (-1.3,-1.3) node[anchor= south west] {$\ddots$};
\end{tikzpicture}
      \end{minipage}
\end{equation}
Write $T_1(n)$, $T_2(N)$ for the two tensors above: in particular 
\[
\multiEMamu(d;n,N) = T_1(n) \boxtimes T_2(N) .
\]
We are going to use the result of Theorem \ref{thm: matrix on many legs} to obtain an upper bound on the exponent $\omega(\multiEMamu(d; n,N))$.

\begin{theorem}\label{thm: omega multiEmamu}
Let $n \geq 0$, $d \geq 3$, $p \in (0,1)$. Then
\[
 \omega( \multiEMamu(d;n,n^{d\frac{1-p}{p}}) ) \leq \frac{1}{p} [\log(n^d + 2n^{d-1} + n^2(n+1)^{d-3} + 1) - h(p)].
\]
\end{theorem}
\begin{proof}
 By Theorem \ref{thm: matrix on many legs}, for every $d \geq 3$ and every $n$, we have the upper bound
 \[
T_1(n) \oplus T_2(n^d) \leq  n^d + 2n^{d-1} + n^2(n+1)^{d-3} + 1.
 \]
 For every $p \in (0,1)$, 
 \[
  T_1(n)^{\boxtimes p} \boxtimes T_2(n^d)^{\boxtimes(1- p)} =  \multiEMamu( d; n^{p},n^{d(1-p)}).
 \]
Therefore Proposition \ref{prop: bound exponent general} provides the upper bound
\[
 \multiEMamu( d; n^{p},n^{d(1-p)}) \leq \log \left[n^d + 2n^{d-1} + n^2(n+1)^{d-3} + 1 \right] - h(p).
\]
Considering the Kronecker power with exponent $1/p$ on the left-hand side, we obtain the desired upper bound.
\end{proof}

The trivial upper bound from \eqref{eqn: trivial asy rank} has the form 
\[
 \omega( \multiEMamu(d;n,n^{d\frac{(1-p)}{p}})) \leq d \log(n) (1 + \textfrac{1-p}{p}).
\]
In the case $p=1/2$, for fixed $d$ and $n$ large, the bound in Theorem \ref{thm: omega multiEmamu} is approximately $2 d \log (n) - 2$, providing a saving of 2 as compared to the trivial upper bound. Note, that this is far away from the lower bound obtained from the flattening lower bound on the asymptotic rank, which is $(2 d -1) \log (n)$.

Let $\omega_{\mathrm{Sch}} ( d,n,p) = \frac{1}{p} (\log (n^d + 2n^{d-1} + n^2(n+1)^{d-3} + 1 )-h(p))$ be the upper bound from Theorem \ref{thm: omega multiEmamu} and let $\omega_{\triv} ( d,n,p) =  d \log(n) (1 + \frac{1-p}{p})$ be the trivial upper bound from \eqref{eqn: trivial asy rank}. 

For $d = 4$, we compare the two upper bounds in Figure \ref{figure: multiemamu d equal 4} for $n = 4 \vvirg 100$ and $p \in (\frac{1}{2},1)$. The new upper bound is nontrivial unless $p$ is close to $1$ as $n$ grows. In particular, we obtain a nontrivial bound for every value of $p < 1/2$.

For $p = \frac{d}{1+d}$, we have $n^{d\frac{(1-p)}{p}} = n$. In Figure \ref{figure: multiemamu p giving n}, we compare the two upper bounds for this value of $p$, for $n = 4 \vvirg 100$ and $d = 3 \vvirg 15$. For fixed $d$, we observe that the new upper bound is nontrivial for $n$ sufficiently large.

Note also that for fixed $d$ and large $n$, the upper bound is approximately  $(d+1) \log(n) - \log (1 + 1/d) \approx (d+1) \log (n) - 1/(d \cdot \mathrm{ln}(2))$, which is strictly lower than the trivial upper bound $(d+1) \log (n)$. However, it is not better than the bound obtained when using the best bounds on the exponent of matrix multiplication, which gives $d-2 + \omega \log (n)$, where $\omega$ is the matrix multiplication exponent.  If $\omega=2$, this matches the trivial flattening lower bound $d \log( n)$.

\begin{figure}[!htp]
\includegraphics[width =.6\textwidth]{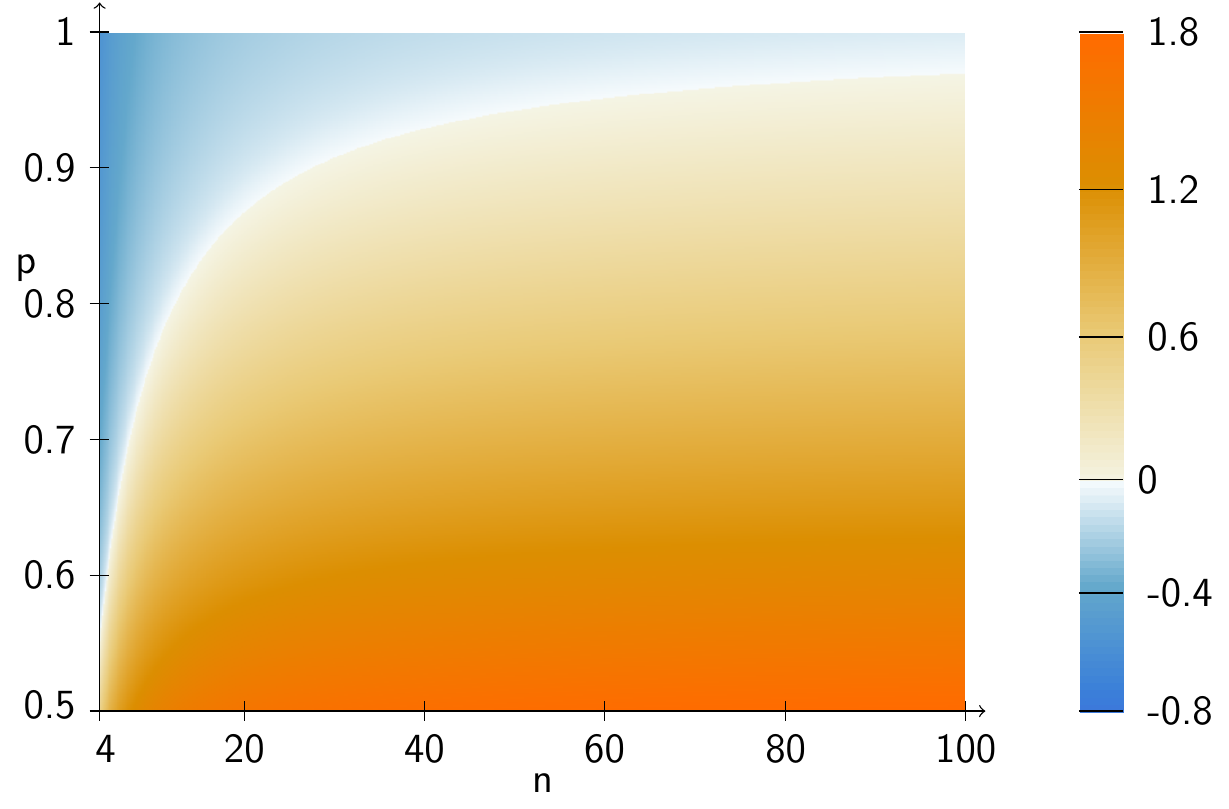}
\caption{Density graph of $\omega_{\triv} - \omega_{\text{Sch}}$ for $d = 4$ as a function of $n$ and $p$. The blue region corresponds to negative values (i.e., $\omega_{\triv} < \omega_{\text{Sch}}$); the orange region corresponds to positive values (i.e., $\omega_{\triv} > \omega_{\text{Sch}}$). Darker shades correspond to more extreme values.} \label{figure: multiemamu d equal 4}
\end{figure}
\begin{figure}[!htp]
\includegraphics[width =.6\textwidth]{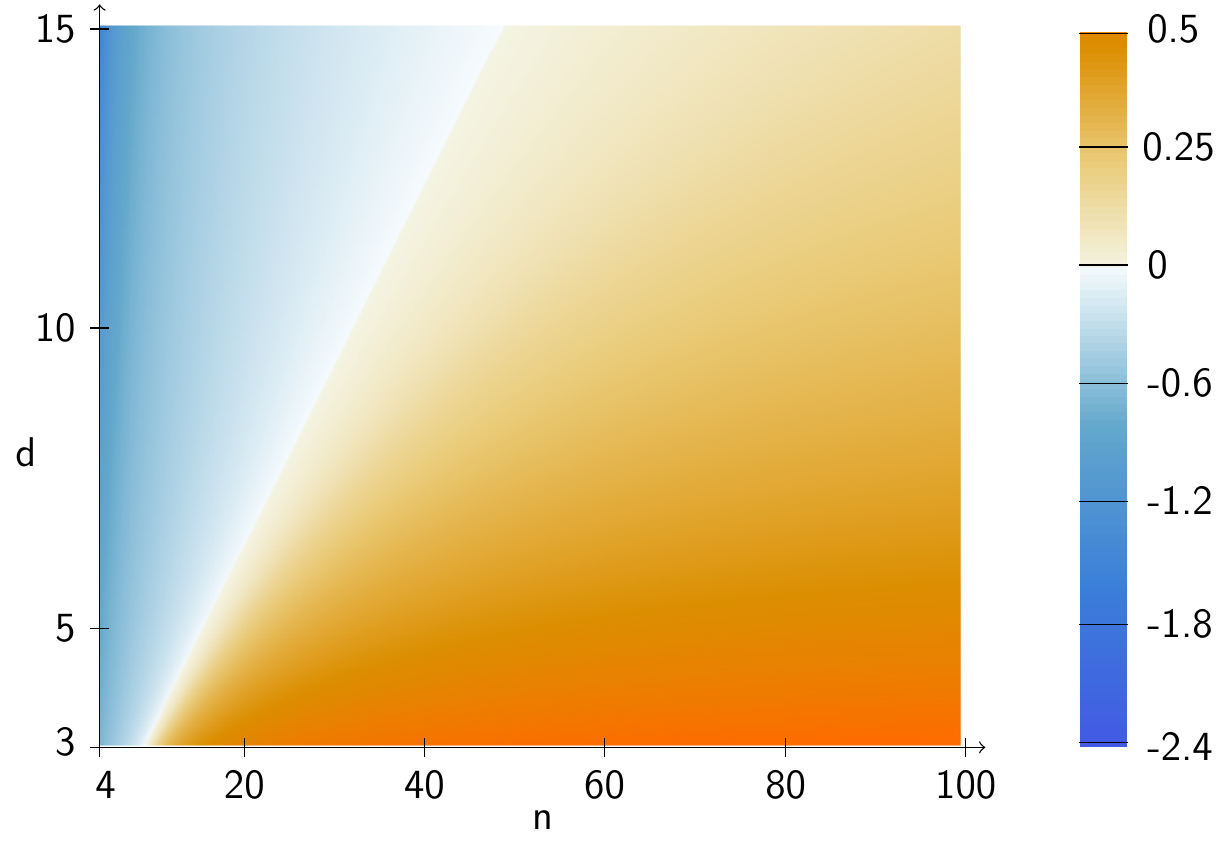}
\caption{Density graph of $\omega_{\triv} - \omega_{\text{Sch}}$ for $p = \frac{d}{d+1}$ as a function of $n$ and $d$. The blue region corresponds to negative values (i.e., $\omega_{\triv} < \omega_{\text{Sch}}$); the orange region corresponds to positive values (i.e., $\omega_{\triv} > \omega_{\text{Sch}}$). Darker shades correspond to more extreme values.} \label{figure: multiemamu p giving n}
\end{figure}

\subsection{Dome tensor}

We use the result of Theorem \ref{thm: absorbing GHZ} to obtain an upper bound on the exponent of the tensor
\[
\Dome(n_1,n_2,n_3;M) \quad = \qquad \begin{minipage}{.3\textwidth}
       \begin{tikzpicture}[scale=1]
 \draw [rounded corners=2mm,fill=gray!50] (1.8,-.07)--(-.07,1.8)--(-1.3,-1.3)--cycle;
 \draw [rounded corners=2mm,fill=white] (1.4,-.04)--(-.04,1.4)--(-0.9,-0.9)--cycle;
\draw[fill=black] (1.5,0) circle (.1cm);
\draw[fill=black] (0,1.5) circle (.1cm);
\draw[fill=black] (0,0) circle (.1cm);
\draw[fill=black] (-1,-1) circle (.1cm);
\draw (1.7,0) node[anchor=west] {$n_1$};
\draw (0,1.7) node[anchor=south] {$n_2$};
\draw (-1,-1) node[anchor= north west ] {$n_3$};
\draw (0,1.5)-- (0,0);
\draw (0,0)-- (1.5,0);
\draw (0,0)-- (-1,-1);
\draw (0.7,-0.8) node[anchor=west] {$M$};
\end{tikzpicture}
\end{minipage}
\]
Following \cite{ChrZui:TensorSurgery}, we call this tensor \emph{dome tensor}. The upper bound from \eqref{eqn: trivial asy rank} provides
\[
 \omega(\Dome(n_1,n_2,n_3; M)) \leq \log(n_1) + \log(n_2) + \log(n_3) + \log(M).
\]
The dome tensor $\Dome(n_1+1,n_2+1,n_3+1;M)$ can be realized as Kronecker product of the tensors $T_1$ and $T_2$ of Construction 4; indeed
\[
\Dome(n_1+1,n_2+1,n_3+1;M) \quad = \qquad \begin{minipage}{.3\textwidth}
       \begin{tikzpicture}[scale=1]
\draw[fill=black] (1.5,0) circle (.1cm);
\draw[fill=black] (0,1.5) circle (.1cm);
\draw[fill=black] (0,0) circle (.1cm);
\draw[fill=black] (-1,-1) circle (.1cm);
\draw (1.7,0) node[anchor=west] {$n_1+1$};
\draw (0,1.7) node[anchor=south] {$n_2+1$};
\draw (-1,-1) node[anchor= north west ] {$n_3+1$};
\draw (0,1.5)-- (0,0);
\draw (0,0)-- (1.5,0);
\draw (0,0)-- (-1,-1);
\end{tikzpicture}
\end{minipage}
\boxtimes
\begin{minipage}{.3\textwidth}
       \begin{tikzpicture}[scale=1]
 \draw [rounded corners=2mm,fill=gray!50] (1.8,-.07)--(-.07,1.8)--(-1.3,-1.3)--cycle;
 \draw [rounded corners=2mm,fill=white] (1.4,-.04)--(-.04,1.4)--(-0.9,-0.9)--cycle;
\draw[fill=black] (1.5,0) circle (.1cm);
\draw[fill=black] (0,1.5) circle (.1cm);
\draw[fill=black] (0,0) circle (.1cm);
\draw[fill=black] (-1,-1) circle (.1cm);
\draw (0.7,-0.8) node[anchor=west] {$M$};
\end{tikzpicture}
\end{minipage}
\]
Fix $n_1=n_2=n_3 = n$; let $T_1(n+1)$ and $T_2(M)$ be the two tensors above, and write $\Dome(n+1;M) := \Dome(n+1,n+1,n+1;M)$; moreover, restrict to the case where $n$ is even, so that Lemma \ref{lemma:best for four factors} holds. We have 
\[
 \Dome(n+1,M) = T_1(n+1) \boxtimes T_2(M).
\]
For $M = \frac{1}{4}n^3$, we are going to use the result of Theorem \ref{thm: absorbing GHZ} to obtain an upper bound on the exponent $\omega(\Dome(n+1,M))$. 

\begin{theorem}\label{thm: upper bound dome}
Let $n \geq 2$ be even, let $p\in (0,1)$. Let $M = \frac{1}{4}n^3$. Then
\begin{equation*}
 \omega \bigl(  \Dome((n+1)^{p}; M^{(1-p)})  \bigr) \leq  \log((n+1)^3+1) - h(p) .
\end{equation*}
\end{theorem}
\begin{proof}
By Theorem \ref{thm: absorbing GHZ}, we have the upper bound $\uR(T_1(n+1) \oplus T_2(M)) = (n+1)^3 + 1$. 

For every $p \in (0,1)$, we have 
\[
 T_1(n+1)^{\boxtimes p} \oplus T_2(M)^{\boxtimes (1-p)} = \Dome ( (n+1)^p, M^{1-p}).
\]
Therefore, Proposition \ref{prop: bound exponent general}, provides the desired upper bound.
 \end{proof}

The trivial upper bound from \eqref{eqn: trivial asy rank} has the form
\begin{equation}\label{eqn: trivial upper bound dome}
\begin{aligned}
 \omega \bigl( \Dome((n+1)^{p}; M^{(1-p)}) \bigr) \leq &3p \log(n+1) + (1-p) \log(M) \\ 
=& 3p \log(n+1) + 3(1-p) \log(n) - 2(1-p),
\end{aligned}
\end{equation}
where we use $M = \frac{1}{4}n^3$.

Let $\omega_{\text{Sch}}(n,p) = \log((n+1)^3+1) - h(p)$ and $\omega_{\text{triv}}(n,p) = 3p \log(n+1) + 3(1-p) \log(n) - 2(1-p)$ be the upper bound from Theorem \ref{thm: upper bound dome} and the trivial upper bound from \eqref{eqn: trivial asy rank}, respectively. We compare the two upper bounds in Figure \ref{figure: exponent dome} for $n = 2 \vvirg 50$ and $p \in (0,1)$. In particular, we observe that for $n$ sufficiently large and $p>1/2$, the upper bound of Theorem \ref{thm: upper bound dome} obtained via the non-additivity construction is stronger than the trivial one. In \cite[Section 4.1]{ChrZui:TensorSurgery}, strong upper bounds on the exponent of some instances of $\omega(\Dome(n,n,n;N))$ are provided, but, this result relies on more advanced methods. On the other hand, the method presented here is extremely simple, and it already provides nontrivial bounds on the exponent, as shown in Figure \ref{figure: exponent dome}.

\begin{figure}[!htp]
 \includegraphics[width =.6\textwidth]{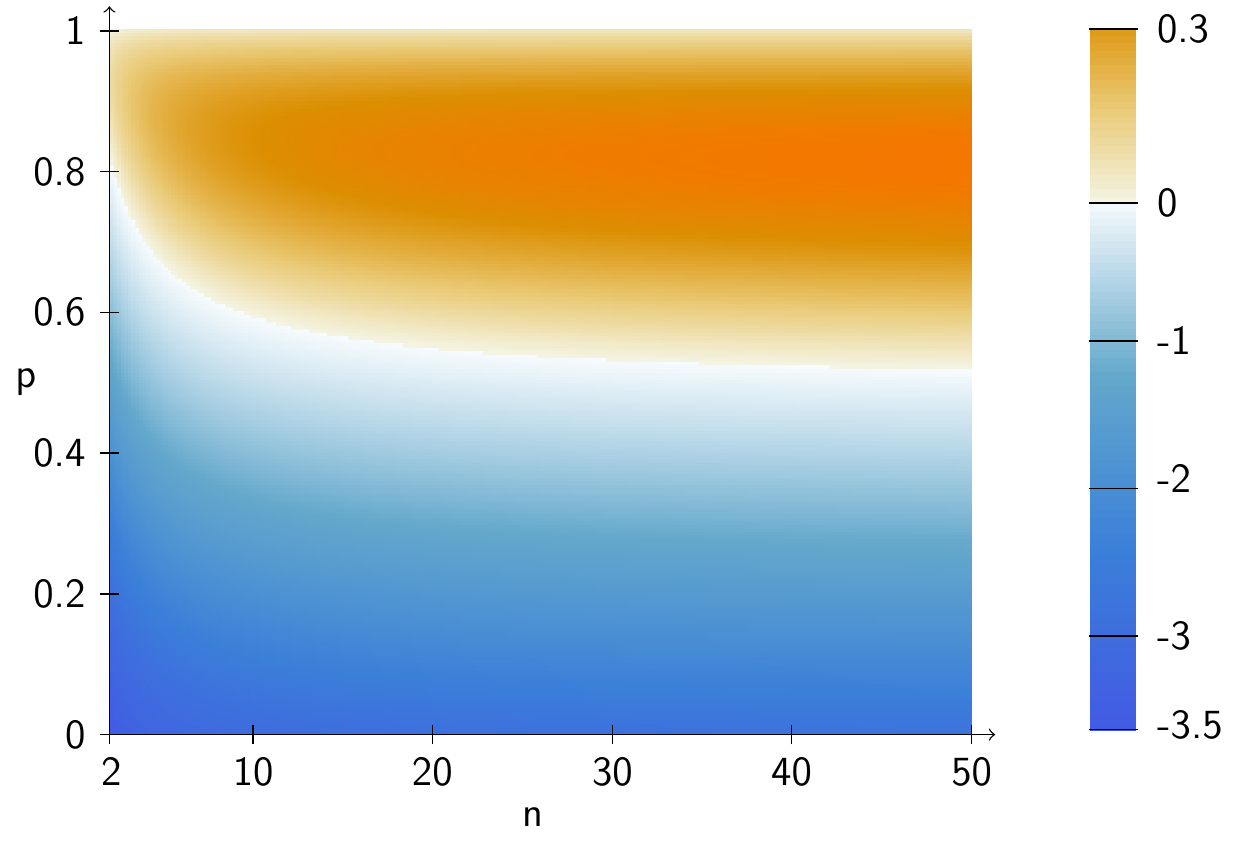}
 \caption{Density graph of $\omega_{\triv} - \omega_{\text{Sch}}$ as a function of $n$ and $p$. The blue region corresponds to negative values (i.e., $\omega_{\triv} < \omega_{\text{Sch}}$); the orange region corresponds to positive values (i.e., $\omega_{\triv} > \omega_{\text{Sch}}$). Darker shades correspond to more extreme values.} \label{figure: exponent dome}
\end{figure}

{\small
\bibliographystyle{plain}
\bibliography{bibAddit}
}
\end{document}